\documentclass[letterpaper, 10 pt, oneside]{amsart}
\usepackage{amsmath,amssymb, amsthm} 
\usepackage{setspace}
\usepackage{marvosym}
\usepackage{hyperref}
\usepackage[utf8]{inputenc}
\hypersetup{
    pdftitle=   {Explicit Bounds for Graph Minors},
   pdfauthor=  {Jim Geelen, Tony Huynh, R. Bruce Richter}
}
\usepackage{color}
\usepackage{enumerate}
\usepackage{pdfsync}
\usepackage{footnote}
 \makeatletter
\newtheorem*{rep@theorem}{\rep@title}
\newcommand{\newreptheorem}[2]{%
\newenvironment{rep#1}[1]{%
 \def\rep@title{#2 \ref{##1}}%
 \begin{rep@theorem}}%
 {\end{rep@theorem}}}
\makeatother

\newtheorem{theorem}{Theorem}[section]
\newreptheorem{theorem}{Theorem}
\newtheorem{lemma}[theorem]{Lemma}
\newtheorem{conjecture}[theorem]{Conjecture}

\newtheorem{claim}[theorem]{Claim}
\newtheorem{case}{Case}
\newtheorem{scl}{Subclaim}
\newtheorem{scase}{Subcase}
\newcommand{\e}{\emph}
\newcommand{\sub}{\subseteq}
\newcommand{\bs}{\backslash}
\newcommand{\D}{\Delta}
\newcommand{\Sg}{\Sigma}
\newcommand{\N}{\mathbb{N}}

\newcommand{\h}{\mathsf{holes}}

\newcommand{\bd}{\mathsf{bd}}

\newcommand{\ex}{\mathsf{ex}}
\newcommand{\clone}{\mathsf{copy}}
\newcommand{\tower}{\mathsf{tower}}

\theoremstyle{definition}
\newtheorem{definition}{Definition}[section]

\begin{document}
\title{Explicit Bounds for Graph Minors}
\author{Jim Geelen}
\author{Tony Huynh}
\author{R. Bruce Richter}
\address[Jim Geelen and R. Bruce Richter]{Department of Combinatorics and Optimization, University of Waterloo, 200 University Avenue West,
Waterloo, ON, N2L 3G1, Canada}
\address[Tony Huynh]{Department of Mathematics,
  Universit\'e Libre de Bruxelles, Avenue Franklin Roosevelt 50, 1050 Brussels,
  Belgium}
\email{jfgeelen@uwaterloo.ca}
\email{tony.bourbaki@gmail.com}
\email{brichter@uwaterloo.ca}
\thanks{This research was partially supported by grants from the Natural Sciences and Engineering Research Council of Canada.  Tony Huynh was also supported by the NWO (The Netherlands Organization for Scientific Research) free
competition project “Matroid Structure – for Efficiency” led by Bert Gerards}

\begin{abstract}
Let $\Sg$ be a surface with boundary $\bd(\Sg)$, $\mathcal{L}$ be a collection of $k$ disjoint $\bd(\Sg)$-paths in $\Sg$, and $P$ be a non-separating $\bd(\Sg)$-path in $\Sg$.  We prove that there is a homeomorphism $\phi: \Sg \to \Sg$ that fixes each point of $\bd(\Sg)$ and such that $\phi(\mathcal{L})$ meets $P$ at most $2k$ times.  

With this theorem, we derive explicit constants in the graph minor algorithms of Robertson and Seymour [Graph minors. XIII. The disjoint paths problem. \e{J. Combin. Theory Ser. B}, 63(1):65–110, 1995]. We reprove a result concerning redundant vertices for graphs on surfaces, but with explicit bounds.  That is, we prove that there exists a \e{computable} integer $t:=t(\Sg,k)$  such that if $v$ is a `$t$-protected' vertex in a surface $\Sg$, then $v$ is redundant with respect to any $k$-linkage.  
\end{abstract}

\maketitle

\section{Introduction}
In~\cite{gm20}, Robertson and Seymour prove the remarkable theorem that every minor-closed property of graphs is characterized by a finite set of excluded minors. 

\begin{theorem}
For every minor-closed class of graphs $\mathcal{C}$, there exists a finite set of graphs $\ex(\mathcal{C})$, such that a graph is in $\mathcal{C}$ if and only if it does not contain a minor isomorphic to a member of $\ex(\mathcal{C})$.  
\end{theorem}

Robertson and Seymour also prove an important algorithmic counterpart to this theorem in~\cite{gm13, gm22}.

\begin{theorem}
For any fixed graph $H$, there exists a polynomial-time algorithm to test if an input graph $G$ contains a minor isomorphic to $H$.
\end{theorem}

Together, these two theorems imply that there \e{exists} a polynomial-time algorithm to test for membership in any minor-closed class of graphs.  Of course, the existence of such an algorithm is highly non-constructive as  $\ex(\mathcal{C})$ is explicitly known for only a few minor-closed classes $\mathcal{C}$.  

The running time of the algorithm from \cite{gm13} depends on a function $t(k, \Sigma)$ for irrelevant vertices for $k$-linkage problems in a surface $\Sg$.  Robertson and Seymour clearly state that $t(k, \Sigma)$ is computable, but give no indication how to compute it.  In the special case that $\Sg$ is the sphere, Adler, Kolliopoulos, Krause, Lokshtanov, Saurabh, and Thilikos~\cite{tightplanar} do obtain an explicit function (of $k$). 

In addition, Kawarabayashi and Wollan~\cite{simpleralgorithm} recently gave a simpler algorithm and shorter proof for the powerful graph minor
decomposition theorem in~\cite{gm16}.  Their approach yields explicit constants for the decomposition algorithm, but again implicitly assumes that $t(k, \Sigma)$ is computable.

In this paper, we show that $t(k, \Sigma)$ is indeed computable, thereby obtaining explicit bounds for graph minors. 
 Before stating our main theorems, we require a few definitions.  In this work we use $\Sg(a,b,c)$ to denote the surface that is the (2-dimensional) sphere with $a$ handles, $b$ crosscaps, and $c$ boundary components, which we call \e{holes}. We set $g(\Sg(a, b, c)) := 2a + b$ and $\h(\Sg(a, b, c)) = c$.

A \e{curve} $\gamma$ in a surface $\Sg$ is a continuous function $\gamma: [0,1] \to \Sg$.  A curve $\gamma$
\begin{itemize}
\item
has \e{ends} $\gamma(0)$ and $\gamma(1)$;

\item
 is a \e{path} if it is injective (or constant);
 
 \item
 is a \e{simple closed curve} if $\gamma(0)=\gamma(1)$ and is injective on $(0,1]$;
 
 \item
 is \e{separating} if $\Sg - \gamma([0,1])$ is disconnected and  \e{non-separating} otherwise.   

\end{itemize}

Let $X \sub \Sg$.  
\begin{itemize}
\item
The boundary and interior of $X$ will be denoted $\bd(X)$ and $\mathsf{int}(X)$, respectively.

\item
A path $\gamma$ is an \e{$X$-path} if the ends of $\gamma$ are in $X$, and $\gamma$ is otherwise disjoint from $X$.    
\end{itemize}

We now define linkages in graphs and in surfaces.  A \e{pattern} $\Pi$ in a graph $G$ is a collection of pairwise disjoint subsets of $V(G)$, where each set in $\Pi$ has size 1 or 2.  

Let $\Pi:=\{ \{s_i, t_i\}  : i \in [k] \}$ be a pattern in $G$ (here $[k]:=\{1, \dots, k\}$ and we allow $s_i=t_i$).  
\begin{itemize}
\item
The \e{vertex set} of $\Pi$ is the set $V(\Pi):=\bigcup \Pi$. 

\item
The \e{size} of $\Pi$ is $|\Pi|=k$.  

\item
 A $\Pi$-linkage in $G$ is a collection $\mathcal{L}:=\{L_1, \dots, L_k\}$ of pairwise disjoint graph-theoretic paths of $G$ where each $L_i$ has ends $s_i$ and $t_i$.  
\end{itemize}
 
Note that if $s_i=t_i$, then $L_i$ is necessarily the path consisting of just the single vertex $s_i$.  

A vertex $v \in V(G)$ is \e{redundant} (with respect to  $\Pi$), provided that $G - v$ has a $\Pi$-linkage if and only if $G$ has a $\Pi$-linkage.  

We  use the same terminology for surfaces.  A \e{pattern} $\Pi$ in a surface $\Sg$ is a collection of pairwise disjoint subsets of $\bd(\Sg)$, each of size 1 or 2.  Let $\Pi:=\{ \{s_i, t_i\}  : i \in [k] \}$ be a pattern in $\Sg$.  A \e{topological $\Pi$-linkage} is a collection $\mathcal{L}:=\{L_1, \dots, L_k\}$ of disjoint $\bd(\Sg)$-paths in $\Sg$ where each $L_i$ has ends $s_i$ and $t_i$.  If $\Sg$ contains a $\Pi$-linkage, we say that $\Pi$ is \e{topologically feasible}. 

Given two linkages $\mathcal{L}$ and $\mathcal{M}$ in a surface $\Sg$, our goal is to perturb $\mathcal{L}$ so that it no longer meets $\mathcal{M}$ very often.  We will only allow a certain 
kind of perturbation of $\mathcal{L}$, which we now define.  
\begin{definition}
A homeomorphism $\phi: \Sg \to \Sg$ is called a \e{$\bd$-homeomorphism}, if $\phi(x)=x$ for each $x \in \bd(\Sg)$.
\end{definition}

We are now prepared to state our first main theorem.  

\begin{theorem} \label{bound}
Let $\Sg$ be a surface and let $\mathcal{L}$ and $\mathcal{M}$ be linkages in $\Sg$ of sizes $k$ and $n$ respectively.  
If $\mathcal{L} \cap \mathcal{M} \cap \bd(\Sg)=\emptyset$ and $\Sg - \mathcal{M}$ is connected, then there is a $\bd$-homeomorphism $\phi: \Sg \to \Sg$ such that $|\phi(\mathcal{L}) \cap \mathcal{M}| \leq k (3^n-1)$.
\end{theorem}

The corresponding result for orientable surfaces (without boundary) was proven by Lickorish~\cite{lickorish}.   
Recently, Matou\v sek, Sedgwick, Tancer and Wagner~\cite{untangling}  considered essentially the same problem.  Using a different approach, they obtain a bound that is 
polynomial in the size of both linkages, while our bound is exponential in the size of one of the linkages (but linear in the other).  

Our proof is shorter than the approach in~\cite{untangling}, but as mentioned, yields different bounds.  Nonetheless, 
Theorem~\ref{bound} appears to be of independent interest. The motivation in~\cite{untangling} comes from an embedding problem involving 3-manifolds.    

To state our second theorem, we need to define the notion of a protected vertex on a surface.  Let $G$ be a graph embedded in a surface $\Sg$ and let $\Pi$ be a pattern in $G$.  

A vertex $v \in V(G)$ is \e{$t$-protected in $\Sg$} (with respect to $\Pi$) if

\begin{itemize}

\item
there are $t$ vertex disjoint cycles $C_1, \dots, C_t$ of $G$, bounding discs $\D_1, \dots, \D_t$ in $\Sigma$ with $v \in \D_1 \subset \D_2 \subset \dots \subset \D_t$, and

\item
$V(\Pi)$ is disjoint from $\mathsf{int}(\D_{t})$.

\end{itemize}

\begin{theorem} \label{protect}
There exists a computable integer $t:=t(\Sg,k)$  such that for all surfaces $
\Sg$ and all $k \in \mathbb{N}$, if $G$ is a graph embedded in $\Sg$, $\Pi$ is a pattern of size $k$ in $G$, and $v \in V(G)$ is a $t$-protected vertex in $\Sg$ with respect to $\Pi$, then $v$ is redundant.
\end{theorem}

We let $\tower(a_1, \dots, a_n)$ be defined inductively as $\tower(a_1)=a_1$ and $\tower(a_1, \dots, a_n)=\tower(a_1, \dots, a_{n-1})^{a_n}$.  The proof of Theorem \ref{protect} shows that we may take $t(\Sg,k)=\tower(100, 200, \dots, 100(4k+3g(\Sg)), k100^{4k+3g(\Sg)})$, although we have not attempted to optimize $t(\Sg,k)$.  Mazoit \cite{mazoit2013} has since simplified our proof of Theorem \ref{protect}, showing that it suffices to take $t(\Sg,k)=C^{k+g(\Sg)}$, for some constant $C$.

The proofs of both of our main theorems do not rely on any of the results in the graph minors series.  

The rest of the paper is organized as follows.  Section~\ref{bounds} contains the proof of Theorem~\ref{bound}. In Section~\ref{main} we derive Theorem~\ref{protect} as a corollary to a slightly different version.  We end by proving the alternative version of Theorem~\ref{protect} in Section~\ref{technical}.

\section{Bounding Intersection Numbers} \label{bounds}
In this section, we prove Theorem~\ref{bound}.  Before starting the proof, we make a few more important definitions.  Let $\Sg$ be a surface and $X$ be a $\bd(\Sg)$-path or a simple closed curve in $\Sg$ disjoint from $\bd(\Sg)$.  We define $\Sg \text{\LeftScissors}X$ to be the surface(s) obtained from $\Sg$ by cutting out a small tubular neighbourhood $\epsilon(X)$ of $X$.  If $X$ is disjoint from some family of curves $\mathcal{C}$ we are considering, we always assume that $\epsilon(X)$ is also disjoint from $\mathcal{C}$.

\begin{definition}
Let $C$ be a simple closed curve in $\Sg$ disjoint from $\bd(\Sg)$. We define $C$ to be 
\begin{itemize}
\item
\e{handle-enclosing}, if a component of $\Sg \text{\LeftScissors}C$ is homeomorphic to $\Sg(1,0,1)$ (a torus with a hole),
\item
\e{crosscap-enclosing}, if a component of $\Sg\text{\LeftScissors}C$ is homeomorphic to $\Sg(0,1,1)$ (a M\"obius band), and
\item
\e{twisted handle-enclosing}, if a component of $\Sg\text{\LeftScissors}C$ is homeomorphic to $\Sg(0,2,1)$ (a Klein bottle with a hole).
\end{itemize}
\end{definition}

\begin{definition}
Two $\bd(\Sg)$-paths $P$ and $P'$ have the same \e{type}, denoted $P \sim P'$, if there is a $\bd$-homeomorphism $\phi$ of $\Sg$ such that $\phi(P)=P'$.  
\end{definition}

Note that for any distinct $x,y \in \bd(\Sg)$, $\sim$ is an equivalence relation on the set of all $\bd(\Sigma)$-paths with ends $x$ and $y$.  
The important thing to note is that there is only a \emph{finite} number of types of $\bd(\Sigma)$-paths with ends $x$ and $y$.   This follows from
the classification theorem for surfaces with holes. 

\begin{definition}
The \e{pseudotype} of a $\bd(\Sg)$-path $P$ is the homeomorphism class of $\Sg \text{\LeftScissors} P$.  
\end{definition}

We now introduce some convenient notation encoding pseudotypes of non-separating $\bd(\Sg)$-paths with ends on the same hole. Let $P$ be such a path.  We say that $P$ is \e{$1$-sided} if $g(\Sg\text{\LeftScissors}P)=g(\Sg)-1$ and $P$ is \e{$2$-sided} if $g(\Sg\text{\LeftScissors}P)=g(\Sg)-2$.
We define $P$ to be \e{orientable} if $\Sg\text{\LeftScissors}P$ is orientable, and \e{non-orientable} otherwise.  
There are only four possible pseudotypes for $P$.  These are determined by the number $i \in [2]$ of sides of $P$ and whether or not $\Sg\text{\LeftScissors}P$ is orientable.  We use the symbols $(i,\rightarrow)$ and $(i,\not\rightarrow)$ to denote that $P$ has $i$ sides and $\Sg\text{\LeftScissors}P$ is or is not orientable, respectively.

The following four lemmas summarize the relevant topological facts connecting types and pseudotypes.  They all follow by cutting along a curve of the prescribed pseudotype and applying the classification theorem for surfaces with boundary.  

\begin{lemma} \label{topology1}
For every orientable surface $\Sg$, any two non-separating $\bd(\Sigma)$-paths with the same ends have the same type.
\end{lemma}

\begin{lemma} \label{topology2}
Let $\Sigma$ be a non-orientable surface and let $x$ and $y$ be distinct points on the same hole of $\bd(\Sigma)$.  
If $P$ and $P'$ are non-separating $\bd(\Sigma)$-paths with ends
$x$ and $y$, then $P$ and $P'$ have the same type if and only if $P$ and $P'$ have the same pseudotype.  
\end{lemma}

\begin{lemma} \label{topology3}
Let $\Sigma$ be a non-orientable surface and let $x$ and $y$ be points on distinct holes $H_x$ and $H_y$ of $\bd(\Sigma)$.
Let $a$ and $b$ be distinct points on $H_x - \{x\}$ and $c$ and $d$ be distinct points on $H_y - \{y\}$.  Let $P_1$ and $P_2$ be $\bd(\Sigma)$-paths
with ends $x$ and $y$ and let $H_i$ be the hole in
$\Sigma \text{\LeftScissors} P_i$ such that $\{a,b,c,d\} \subseteq H_i$.  Then $P_1$ and $P_2$ have the same type if and only if $\{a,b,c,d\}$ has the same
cyclic order in $H_1$ and $H_2$.  
\end{lemma}

The previous three lemmas completely describe when two non-separating paths are of the same type.  The next lemma classifies types of
separating paths.

\begin{lemma} \label{topology4}
Let $\Sigma$ be a surface, $x$ and $y$ be distinct points on the same hole $H$ of $\bd(\Sigma)$, and $P$ and $P'$ be separating $\bd(\Sg)$-paths with ends $x$ and $y$.
Then $P$ and $P'$ have the same type if and only if there exists an ordering $\Sigma_1, \Sigma_2$ of the components of $\Sigma \text{\LeftScissors} P$ and an ordering  $\Sigma_1', \Sigma_2'$  of the components of $\Sigma \text{\LeftScissors} P'$ so that for $i=1,2$,
$\Sigma_i \cong \Sigma_i'$  and
$\Sigma_i \cap \bd(\Sigma)=\Sigma_i' \cap \bd(\Sigma)$.
\end{lemma}

\begin{definition}
A path $P$ in a surface $\Sg$ is \e{contractible} if $P$ is a $\delta$-path for some hole $\delta$ of $\Sg$ and some component of $\Sg-P$ is an open disk.
\end{definition}

\begin{definition}
Two $\bd(\Sg)$-paths are \e{homotopic} if there is a homotopy between them that always has its endpoints on $\bd(\Sg)$.  
\end{definition}
%
%\begin{lemma} \label{pseudohomotopic}
%Let $\Sg$ be a surface, $\delta$ be a hole of $\Sg$, and $L_1$ and $L_2$ be disjoint $\delta$-paths in $\Sg$.  If $L_1$ and $L_2$ are both of pseudotype $(1, \to)$ or both of pseudotype $(2, \to)$, then $L_1$ and $L_2$ are homotopic.  
%\end{lemma}
%
%\begin{proof}
%Observe that $L_2$ is a contractible path in the surface $\Sg-L_1$.  Therefore, there is a component of $\Sg - (L_1 \cup L_2)$ that is a disk.  Hence $L_1$ and $L_2$ are homotopic, as required.  
%\end{proof}

The final definition we require concerns intersection numbers of curves.  

\begin{definition}
The \e{geometric intersection} of a $\bd(\Sg)$-path $P_1$ with a $\bd(\Sg)$-path $P_2$ is defined to be 
\[
\#(P_1,P_2):=\min \{|P_1 \cap P_2'|: \text{$P_2'$ is of the same type as $P_2$} \}.
\]
\end{definition}

Note that for any two $\bd(\Sg)$-paths $P_1$ and $P_2$, we have $\#(P_1,P_2) \leq 2$ by the previous lemmas.  Furthermore, in an orientable surface $\Sg$, the type of a non-separating $\bd(\Sg)$-path is determined by its pseudotype.  Therefore, the following lemma follows by an easy case analysis.  

\begin{lemma} \label{orientable0}
If $\Sg$ is an orientable surface and $P_1$ and $P_2$ are non-separating $\bd$-paths in $\Sg$ with $P_1 \cap P_2 \cap \bd(\Sg)=\emptyset$, then $\#(P_1, P_2)=0$.  
\end{lemma}

Now that the topological
prerequisites are in place, we proceed to prove Theorem~\ref{bound}.  We first consider the special case that $|\mathcal{M}|=1$.  Theorem~\ref{bound} will then 
follow by induction.

\begin{theorem} \label{specialbound}
Let $\Sg$ be a surface and let $P$ be a non-separating $\bd(\Sg)$-path in $\Sg$.  For any  linkage $\mathcal{L}$ in $\Sg$ whose ends are disjoint from $P$, there is a $\bd$-homeomorphism $\phi: \Sg \to \Sg$ such that each path of $\phi(\mathcal{L})$ intersects $P$ at most twice.   
\end{theorem}

\begin{proof}
We define an \e{$(\mathcal{L}, P)$-shift} to be a $\bd$-homeomorphism $\phi: \Sg \to \Sg$ such that each path of $\phi(\mathcal{L})$ intersects $P$ at most twice.   
Let $(\Sg, P, \mathcal{L})$ be a counterexample with $(g(\Sg), \h(\Sg), |\mathcal{L}|)$ lexicographically minimal.  

We proceed by establishing a chain of claims.  To begin, even though we only care about the theorem when $P$ is non-separating, for inductive purposes it is helpful to note that it holds in the following special case when $P$ is separating.  

\begin{claim} \label{diskhole}
If $P$ is contractible, then there is an $(\mathcal{L}, P)$-shift.  
\end{claim}

\begin{proof}[Subproof]
There is an isotopy $\phi:\Sigma\to\Sigma$ (fixing each point of $\bd(\Sigma)$) that moves $P$ sufficiently close to $\bd(\Sigma)$ so that each $L\in \mathcal L$ meets $\phi(P)$ only near an end of $L$.  Therefore, $|\phi(P)\cap L|\leq 2$.  In this case, $\phi^{-1}$ is an $(\mathcal{L}, P)$-shift.
\end{proof}

Similarly, we have the following.

\begin{claim} \label{ncontract}
No $L \in \mathcal{L}$ is contractible.
\end{claim}

\begin{proof}[Subproof]
Suppose $\mathcal{L}$ contains a contractible path.  Since the paths in $\mathcal{L}$ are disjoint, there must be a path $L \in \mathcal{L}$ such that one component of $\Sg \text{\LeftScissors} L$ is an open disk which is disjoint from $\mathcal{L}$. Consider $\mathcal{L} - L$ in $\Sg$.  By minimality, there exists an $(\mathcal{L}-L, P)$-shift $\phi$.  If $\phi(L)$ also meets $P$ at most twice we are done.  
Next observe that $\Sigma \text{\LeftScissors} \phi(L)$ has a component $\phi(\Delta)$ such that $\phi(\Delta)$ is an open disk disjoint from $\phi(\mathcal{L})$.  Thus, we 
may apply an isotopy $\alpha: \Sigma \to \Sigma$ to shift $L$ near $\bd(\Sg)$ so that $|\phi(L') \cap P|=|\alpha \phi (L') \cap P|$ for all $L' \in \mathcal{L} -L$ and $|\alpha \phi (L) \cap P| \leq 2$.  
\end{proof}

%\begin{claim} \label{nonhomo}
%No two paths in $\mathcal{L}$ are homotopic.
%\end{claim}
%
%\begin{proof}[Subproof]
%Suppose not.  By the previous claim, no $L \in \mathcal{L}$ is contractible.  Hence, 
%there must exist two distinct homotopic paths $L_1$ and $L_2$ of $\mathcal{L}$ such that one component $\D$ of $\Sg - (L_1 \cup L_2)$ is an open disk disjoint from $\mathcal{L}$.  
%Consider $\mathcal{L} - L_2$ in $\Sg$.  By minimality, there is a $\bd$-homeomorphism $\phi: \Sg \to \Sg$  such that each path of $\phi(\mathcal{L}-L_2)$ meets $P$ at most twice.    We finish by applying an isotopy $\alpha: \Sigma \to \Sigma$ so that $|\phi(L') \cap P|=|\alpha \phi (L') \cap P|$ for all $L' \in \mathcal{L} - L_2$ and $|\alpha \phi (L_2) \cap P| =|\phi(L_1) \cap P|$.  
%\end{proof}

\begin{claim} \label{one}
For all $L \in \mathcal{L}, \#(P, L) \neq 0$. 
\end{claim}

\begin{proof}[Subproof]
Towards a contradiction, assume that $\#(P,L')=0$ for some $L' \in \mathcal{L}$. Let $\phi: \Sg \to \Sg$ be a $\bd$-homeomorphism such that $\phi(L')$ is disjoint from $P$.  Let $\Sg'$ be the component of $\Sg \text{\LeftScissors} \phi(L')$ that contains $P$ (possibly $\Sg'=\Sg \text{\LeftScissors} \phi(L')$).  Consider the linkage $\mathcal{L}':=\phi(\mathcal{L} - L) \cap \Sg'$.  Since $(\Sg, P, \mathcal{L})$ is a minimal counterexample, there exists an $(\mathcal{L}', P)$-shift $\alpha: \Sg' \to \Sg'$. Consider the map $\beta: \Sigma \to \Sigma$ defined by $\beta(x):=\alpha \phi (x)$ if $x \in \phi^{-1}(\Sg')$ and $\beta(x):=\phi(x)$ otherwise.  By construction, $\beta$ is an $(\mathcal{L}, P)$-shift, which is a contradiction. 
\end{proof}

\begin{claim} \label{two}
If $L \in \mathcal{L}$ is separating, then $\#(P,L)=2$.  
\end{claim}

\begin{proof}
Let $L' \in \mathcal{L}$ be a separating curve and let $\Sg_1$ and $\Sg_2$ be the two components of $\Sg \text{\LeftScissors} L'$.  
By the previous claim, we know that $\#(P,L') \neq 0$.  Towards a contradiction, suppose that $\#(P,L')=1$.   By Lemma~\ref{topology1}, Lemma~\ref{topology2} or Lemma~\ref{topology3}, we may choose a curve $P'$ of the same type as $P$ such that $|P' \cap L'|=1$ and for $i\in \{1,2\}$, $P' \cap \Sg_i$ is either non-separating or contractible in $\Sg_i$.  

 Let $\phi: \Sg \to \Sg$ be a $\bd$-homeomorphism such that $\phi(P)=P'$.  Note that $\phi^{-1}(L')$ only intersects $P$ once. Let $\Sg_1'$ and $\Sg_2'$ be the two components of $\Sg \text{\LeftScissors} \phi^{-1}(L')$.  By Claim~\ref{diskhole} and induction, there are $\bd$-homeomorphisms $\alpha_i: \Sg_i' \to \Sg_i'$ such that each path of $\alpha_i (\phi^{-1} (\mathcal{L}) \cap  \Sg_i')$ meets $P \cap \Sg_i'$ at most twice in $\Sg_i'$.  Thus, by combining $\alpha_1 \phi^{-1}$ and $
 \alpha_2 \phi^{-1}$ appropriately, we obtain an $(\mathcal{L}, P)$-shift.
\end{proof}

\begin{claim} \label{Pbyitself}
No path in $\mathcal{L}$ intersects any hole that $P$ intersects.  
\end{claim}

\begin{proof}[Subproof]
Suppose not and let $\delta$ be a hole such that both $P$ and $\mathcal{L}$ meet $\delta$.  There must exist a path $L' \in \mathcal{L}$ such that one end  $l$ of $L'$ and one end $p$ of $P'$ are \e{consecutive along $\delta$}. That is, there is a component of $\delta - \{l,p\}$ that is disjoint from $\mathcal{L} \cup P$.  Note that $\#(P, L')=1$ if $L'$ is non-separating, and $\#(P,L')=2$ if $L'$ is separating.  We will handle both possibilities simultaneously. 

Let $\phi: \Sg \to \Sg$ be a $\bd$-homeomorphism such that $|\phi(L') \cap P| = \#(P, L')$. Let $\Sg_1$ and $\Sg_2$ be the components of $\Sg \text{\LeftScissors} \phi(L')$ (we allow $\Sg_2 = \emptyset$, in case $L'$ is non-separating).  Consider $P \cap \Sg_1$ and $P \cap \Sg_2$.  Relabelling $\Sg_1$ and $\Sg_2$ if necessary, we may assume that $P \cap \Sg_1$ consists of two disjoint subpaths $P_1$ and $P_1'$ of $P$ and $P  \cap \Sg_2$ is a single (possibly empty) subpath $P_2$ of $P$.  Since $l$ and $p$ are consecutive along $\delta$ we may also assume that one component $\D$ of $\Sg_1 -P_{1}'$ is a disk which is disjoint from $\phi(\mathcal{L})$. 

As neither $(\Sg_1, P_{1}, \phi(\mathcal{L} \cap \Sg_1))$ nor  $(\Sg_2, P_{2}, \phi(\mathcal{L} \cap \Sg_2))$ are counterexamples, there exist $\bd$-homeomorphisms $\alpha_i: \Sg_i \to \Sg_i$ such that each path of $\alpha_i (\phi(\mathcal{L} \cap \Sg_i))$ meets $P_{i}$ at most twice.  Note that it is possible that $\alpha_1 (\phi(\mathcal{L} \cap \Sg_1))$ intersects $P_{1}'$.  However, as $\D$ is disjoint from $\phi(\mathcal{L})$, there is an isotopy $\gamma: \Sg_1 \to \Sg_1$ such that $\gamma\alpha_1 (\phi(\mathcal{L} \cap \Sg_1))$ does not meet $P_1'$ and  $|\gamma(L) \cap P_1|=|L \cap P_1|$ for all paths $L \in  \alpha_1 (\phi(\mathcal{L} \cap \Sg_1))$. If we now define $\beta: \Sg \to \Sg$ by 
$$
\beta(x)=
\begin{cases}
\gamma \alpha_1 \phi (x), & \text{ if $x \in \phi^{-1}(\Sg_1)$} \\
\alpha_2\phi (x), & \text{ if $x \in \phi^{-1}(\Sigma_2)$} \\
\phi(x),  & \text{ otherwise}
\end{cases}
$$

we contradict that $(\Sg, P, \mathcal{L})$ is a counterexample.  
\end{proof}

%\begin{claim} \label{alldisks}
%Every component of $\Sigma - (P \cup \mathcal{L})$ is a disk.  
%\end{claim}
%
%\begin{proof}[Subproof]
%If not, then then there is a non-separating curve $C$, such that $C$ is disjoint from $P \cup \mathcal{L}$.
%Cutting along a sufficiently small tubular neighbourhood of $C$ produces a smaller counterexample than $(\Sg, P, L)$.  
%\end{proof}

\begin{claim} \label{nonseparating}
Each $L \in \mathcal{L}$ is non-separating.
\end{claim}

\begin{proof}
Suppose that $L \in \mathcal{L}$ is separating.  By Claim~\ref{two}, $\#(P,L)=2$.  In particular, this implies that both ends of $P$ are on the same hole $\delta$.  
 Let $\phi: \Sg \to \Sg$ be a $\bd$-homeomorphism such that $|\phi(L) \cap P|=2$.  Let $\Sg_1$ and $\Sg_2$ be the two components of $\Sg \text{\LeftScissors} \phi(L)$.  
We may assume that $\Sg_1 \cap P$ consists of two disjoint subpaths $P_1$ and $P_1'$ of $P$ and $\Sg_2 \cap P$ is a single subpath $P_2$ of $P$. 

By Claim~\ref{Pbyitself}, $\delta$ is disjoint from $\mathcal{L}$.  Therefore, by Lemma~\ref{topology4}, we may assume that $P_1$ and $P_1'$ connect different holes of $\Sg_1$ and that $P_1$ and $P_1'$ are homotopic in $\Sg_1$.

As neither $(\Sg_1, P_{1}, \phi(\mathcal{L} \cap \Sg_1))$ nor  $(\Sg_2, P_{2}, \phi(\mathcal{L} \cap \Sg_2))$ are counterexamples, there exist $\bd$-homeomorphisms $\alpha_i: \Sg_i \to \Sg_i$ such that each path of $\alpha_1 (\phi(\mathcal{L} \cap \Sg_i)$ meets $P_{i}$ at most twice.  If $\alpha_1 (\phi(\mathcal{L} \cap \Sg_1))$ intersects $P_{1}'$ at most twice, then we are done by combining $\alpha_i$ and $\phi$ appropriately.  Otherwise, since $P_1$ and $P_1'$ are homotopic in $\Sg_1$ and $\mathcal{L}$ is disjoint from $\delta$, there is a component $\D$ of $\Sg_1 - (P_1 \cup P_1')$ that is an open disk disjoint from $\phi(\mathcal{L})$. Therefore, we are done by applying an appropriate isotopy of $\Sg_1$.
\end{proof}

\begin{claim} \label{nonorientable}
$\Sg$ is non-orientable.
\end{claim}

\begin{proof}[Subproof]
Arbitrarily choose $L \in \mathcal{L}$.  By the previous claim, $L$ is non-separating.  If $\Sg$ is orientable, then $\#(P,L)=0$, by Lemma~\ref{orientable0}.  This contradicts Claim~\ref{one}.  
\end{proof}

\begin{claim} \label{samehole}
No member of $\mathcal{L} \cup \{P\}$ has endpoints on distinct holes of $\Sg$.
\end{claim}

\begin{proof}[Subproof]
Arbitrarily choose $L \in \mathcal{L}$.  By Claim~\ref{nonseparating} and Claim~\ref{Pbyitself}, $L$ is non-separating and neither end of $L$ is on the same hole as an end of $P$.  Therefore, if $L$ or $P$ has endpoints on distinct holes, then $\#(P,L)=0$, a contradiction. 
\end{proof}

We finish the proof by ruling out all four possibilities for the pseudotype of $P$.  Let $\mathcal{L}:=\{L_1, \dots, L_n\}$, let $p_1$ and $p_2$ be the ends of $P$, and let $\delta_P$ be the hole which contains $\{p_1, p_2\}$.  By Claim~\ref{samehole}, each $L_i$ is also a $\delta_i$-path for some hole $\delta_i$.   Also, by Claim~\ref{Pbyitself}, $\delta_i \neq \delta_P$ for any $i$.

\begin{claim}
$P$ is not of pseudotype $(2, \not\to)$.
\end{claim}

\begin{proof}[Subproof]
Suppose $P$ is of pseudotype $(2, \not\to)$. This implies that $\Sg \cong \Sg(0,i,j)$ for some $i \geq 3$. Let $C$ be a separating curve such that one component $\Sg_1$ of $\Sg \text{\LeftScissors}C$ is homemorphic to $\Sg(1,0,2)$ and $P \subseteq \Sg_1$.  Let $\Sg_2$ be the other component of $\Sg\text{\LeftScissors}C$.  Note that $\Sg_2 \cong \Sg(0,i-2, j)$.   We choose an arbitrary $L \in \mathcal{L}$ and show in every case that we get the contradiction $\#(P,L) = 0$.   

If $L$ has pseudotype $(1, \to)$ or $(1, \not\to)$, then there is a path of the same type as $L$ contained in $\Sg_2$, and hence disjoint from $P$.
If $L$ has pseudotype $(2,\to)$, then $i$ is even and at least 4, so again there is a path of the same type as $L$ contained in $\Sg_2$.   If $L$ is of pseudotype $(2, \not \to)$, then there is a path of the same type as $L$ disjoint from $P$ that meets $C$ exactly twice.  
\end{proof}

\begin{claim} \label{1to}
$P$ is not of pseudotype $(1, \to)$.
\end{claim}

\begin{proof}[Subproof]
If $P$ is of pseudotype $(1,\to)$, then $\Sigma \cong \Sigma(i,1,j)$, for some $i,j$.
  Consider an arbitrary $L \in \mathcal{L}$.  Since $g(\Sigma) $ is odd, $L$ is not of type $(2,\rightarrow)$.   Observe that $L$ cannot be of type $(2,\not\rightarrow)$, as otherwise $g(\Sigma)\ge 3$ and $\#(P,L)=0$.  Hence, each $L_k$ is of pseudotype $(1, \to)$ or $(1, \not\to)$. 

Let $C_0, C_1, \dots, C_i$ be disjoint closed curves in $\Sg$ such that $C_0$ is a crosscap-enclosing curve and $C_1, \dots, C_i$ are pairwise non-homotopic handle-enclosing curves.  Since each path in $\mathcal{L}$ is of pseudotype $(1, \to)$ or $(1, \not\to)$, each path in $\mathcal{L}$ must intersect $C_0$.  By applying an appropriate isotopy, we may assume that each $L_k$ intersects $C_0$ exactly twice.  Thus, we may label the points of $C_0 \cap \mathcal{L}$ as $x_1, x_1', \dots, x_n, x_n'$, where $x_k$ and $x_k'$ are the ends of $L_k$, and the clockwise order of $C_0 \cap \mathcal{L}$ along $C_0$ is $x_1, \dots, x_n, x_1', \dots, x_n'$.  

Since each $L_k$ is non-separating, there is a path $Q$ in $\Sg$ from $p_1$ to a point $z \in C_0$ that avoids $L_1 \cup \dots \cup L_n \cup C_0 \cup C_1 \cup \dots \cup C_i$ (other than the point $z$).  Let $\Sg_0$ be the crosscap enclosed by $C_0$.  By relabelling if necessary, we may assume that there is a point $z' \in C_0$ such that the clockwise order of $\{z, z', x_1, \dots, x_n, x_1', \dots, x_n'\}$ along $C_0$ is
$z, x_1, \dots, x_n, z', x_1', \dots, x_n'$.  
Thus, there is a path $R$ in $\Sg_0$ such that $R \cap C_0:=\{z,z'\}$ and $R$ is disjoint from $\mathcal{L}$.

We now define a path $P'$ with the same ends as $P$
as follows. 

\begin{itemize}
\item
Start at $p_1$ and follow $Q$ until reaching $z$. 
\item
Follow $R$ until reaching $z'$. 
\item
Follow $C_0$ clockwise until returning sufficiently close to $z$.  
\item
Stay sufficiently close to $Q$ until returning sufficiently close to $p_1$.
\item
Stay sufficiently close to $\delta_P$ until returning to $p_2$.
\end{itemize}

Since $\delta_P$ does not meet any $L_k$, we may choose $P'$ so that $P'$ meets each $L_k$ exactly once.  Moreover, we may also assume that $P'$ does not meet $C_1 \cup \dots \cup C_i$.  Therefore, by construction, $P'$ is of pseudotype $(1, \to)$.  By Lemma~\ref{topology2}, $P'$ is of the same type as $P$, so we are done.  
\end{proof}
 
\begin{claim}
$P$ is not of pseudotype $(1, \not\to)$.
\end{claim}

\begin{proof}[Subproof]
Suppose not and consider an arbitrary $L \in \mathcal{L}$.  Observe that $\#(P,L)=0$, unless $L$ is of pseudotype $(1, \to)$ or $(2, \to)$.  Therefore, each path in $\mathcal{L}$
is of pseudotype $(1, \to)$ if $g(\Sg)$ is odd, or each path in $\mathcal{L}$ is of pseudotype $(2, \to)$ if $g(\Sg)$ is even.  

We handle the former possibility first.  In this case $\Sg$ is homeomorphic to $\Sg(0, 2i+1, j)$ for some $i,j$.  Let $C_0, C_1, \dots, C_{2i}$ be pairwise disjoint non-homotopic crosscap-enclosing curves in $\Sg$.  Since each path in $\mathcal{L}$ is of pseudotype $(1, \to)$, each path in $\mathcal{L}$ must intersect $C_0$.  By applying an appropriate isotopy, we may assume that each $L_k$ intersects $C_0$ exactly twice.  Now as in the proof of Claim~\ref{1to}, we can construct a path of the same type as $P$ which meets each curve in $\mathcal{L}$ exactly once.

The remaining case is if each $L \in \mathcal{L}$ is of pseudotype $(2, \to)$, which implies that $\Sg \cong \Sg(i,2,j)$ for some $i,j$.  
Let $C_0, C_1, \dots, C_i$ be disjoint closed curves in $\Sg$ such that $C_0$ is a twisted handle-enclosing curve and $C_1, \dots, C_i$ are pairwise non-homotopic handle-enclosing curves.  Observe that each path in $\mathcal{L}$ must intersect $C_0$.  By applying an appropriate isotopy, we may assume that each $L_k$ intersects $C_0$ exactly twice.  Thus, we may label the points of $C_0 \cap \mathcal{L}$ as $x_1, x_1', \dots, x_{n}, x_{n}'$, where $x_k$ and $x_k'$ are the ends of $L_k$, and the clockwise order of $C_0 \cap \mathcal{L}$ along $C_0$ is $x_1, \dots, x_{n}, x_n', \dots, x_{1}'$.  Let $y$ and $y'$ be points of $C_0$ such that the clockwise order
of $\{y, y'\} \cup (C_0 \cap \mathcal{L})$ along $C_0$ is $x_1, \dots, x_{n}, y, x_n', \dots, x_{1}', y'$.  

In this case, we start at $p_1$ until we get nearly to $C_0$ at some point $z$; follow along $C_0$ to $y$ or $y'$, go through the twisted handle, then back alongside $C_0$ to near $z$, and finish as in Claim~\ref{1to}.
 \end{proof}

\begin{claim}
$P$ is not of pseudotype $(2, \to)$.
\end{claim}

\begin{proof}[Subproof]
Suppose not and note $\Sg \cong \Sg(i,2,j)$ for some $i \geq 0$.  Observe that $L$ cannot be of pseudotype $(2, \not\to)$ or $(2, \to)$, otherwise $\#(P,L)=0$.  Therefore, each $L_k$ is of pseudotype $(1, \not\to)$. 

Let $C_0, C_1, \dots, C_i$ be disjoint closed curves in $\Sg$ such that $C_0$ is a twisted handle-enclosing curve and $C_1, \dots, C_i$ are pairwise non-homotopic handle-enclosing curves.  Since each path in $\mathcal{L}$ is of pseudotype $(1, \not\to)$, $\mathcal{L}$ must intersect $C_0$.  By applying an appropriate isotopy, we may assume that each $L_k$ intersects $C_0$ exactly twice.  Note that some paths of $\mathcal{L}$ go through one of the crosscaps enclosed by $C_0$, and the rest must go through the other crosscap enclosed by $C_0$.  Thus, we may label the points of $C_0 \cap \mathcal{L}$ as $x_1, x_1', \dots, x_{n_1}, x_{n_1}', y_1, y_1' \dots, y_{n_2}, y_{n_2}'$, where $x_k$ and $x_k'$ are the ends of $L_k$, $y_k$ and $y_k'$ are the ends of $L_{n_1+k}$, $n_1+n_2=n$, and the clockwise order of $C_0 \cap \mathcal{L}$ along $C_0$ is 
\[x_1, \dots, x_{n_1}, x_1', \dots, x_{n_1}', y_1, \dots, y_{n_2},   
y_1', \dots, y_{n_2}'.
\]  

Again there is a path from $p_1$ to a point $z \in C_0$ that avoids $\mathcal{L} \cup C_1 \cup \dots \cup C_i$.  By symmetry we may assume that $z$ is on the clockwise segment of $C_0$ from $x_1$ to $x_{n_1}'$.   Now let $w$ and $w'$ be points of $C_0$ such that the clockwise order  of $\{w,w'\} \cup (\mathcal{L} \cap C_0)$  along $C_0$ is  
\[x_1, \dots, x_{n_1}, x_1', \dots, x_{n_1}', w, y_1, \dots, y_{n_2},  w', 
y_1', \dots, y_{n_2}'.
\]  
In this case, we start at $p_1$ until we get nearly to $C_0$ at $z$; go through one of the crosscaps enclosed by $C_0$, then alongside $C_0$ to $w$ or $w'$, then through the other crosscap enclosed by $C_0$, then back alongside $C_0$ until returning to near $z$, and finish as in Claim~\ref{1to}.
\end{proof}

This completes the entire proof.
\end{proof}

A simple induction yields Theorem~\ref{bound}, which is the form we will use later.

\begin{reptheorem}{bound}
Let $\Sg$ be a surface and let $\mathcal{L}$ and $\mathcal{M}$ be linkages in $\Sg$ of sizes $k$ and $n$ respectively.  
If $\mathcal{L} \cap \mathcal{M} \cap \bd(\Sg)=\emptyset$ and $\Sg - \mathcal{M}$ is connected, then there is a $\bd$-homeomorphism $\phi: \Sg \to \Sg$ such that $|\phi(\mathcal{L}) \cap \mathcal{M}| \leq k (3^n-1)$.
\end{reptheorem}

\begin{proof}
We proceed by induction on $n$.  The case $n=1$ follows by the previous theorem.  Let $P \in \mathcal{M}$.  By the previous theorem, there is a $\bd$-homeomorphism $\phi_1: \Sg \to \Sg$ such that each path of $\phi_1(\mathcal{L})$ intersects $P$ at most twice. Let $\Sigma'$  and $\mathcal{L}'$ be the surface and linkage obtained from $\Sigma$ and $\phi_1(\mathcal{L})$ by cutting out a small tubular neighbourhood of $P$. 
Thus, $\mathcal{L}'$ is a linkage in $\Sigma'$ of size at most $3k$.  By induction, there is a $\bd$-homeomorphism $\phi_2: \Sg' \to \Sg'$ such that $|\phi_2(\mathcal{L}') \cap (\mathcal{M} \setminus \{P\})| \leq (3k) (3^{n-1}-1)$. Thus, there is a $\bd$-homeomorphism $\phi: \Sg \to \Sg$ such that $|\phi(\mathcal{L}) \cap \mathcal{M}| \leq (3k) (3^{n-1}-1)+2k=k (3^n-1)$.
\end{proof} 

We conjecture that Theorem \ref{bound} holds without the assumption that $\Sg - \mathcal{M}$ is connected.  

\begin{conjecture} \label{separatingconjecture}
Let $\Sg$ be a surface and let $\mathcal{L}$ and $\mathcal{M}$ be linkages in $\Sg$ of sizes $k$ and $n$ respectively.  
If $\mathcal{L} \cap \mathcal{M} \cap \bd(\Sg)=\emptyset$, then there is a $\bd$-homeomorphism $\phi: \Sg \to \Sg$ such that $|\phi(\mathcal{L}) \cap \mathcal{M}| \leq k (3^n-1)$.
\end{conjecture}

We end the section by connecting Theorem~\ref{bound} to a constant $w(\Sigma,k,n)$ that appears in Graph Minors VII \cite{gm7}.  A \emph{near-linkage} in a surface $\Sg$ is a collection of internally disjoint $\bd(\Sg)$-paths.  If $\Sg$ is a cylinder, then a small subset of the proofs of Theorems \ref{specialbound} and \ref{bound} yields the following. 

\begin{theorem} \label{nearcylinder}
Let $\mathcal{L}$ be a near-linkage of size $k$ and $\mathcal{M}$ be a linkage of size $n$ in a cylinder $\Sg$.  Then there is a $\bd$-homeomorphism $\phi: \Sg \to \Sg$ such that $|\phi(\mathcal{L}) \cap \mathcal{M}| \leq k (3^n-1)$. 
\end{theorem}

Note that in the cylinder, we do not require the hypotheses $\mathcal{L} \cap \mathcal{M} \cap \bd(\Sg)=\emptyset$ nor $\Sg-\mathcal{M}$ connected. The latter follows easily since every separating $\bd$-path $P$ in the cylinder is contractible, and we know that Theorem \ref{specialbound} holds if $P$ is contractible.  We now extend Theorem \ref{bound} to the case that $\mathcal{L}$ is a near-linkage.

\begin{theorem}
\label{nearlinkage}
Let $\Sg$ be a surface, $\mathcal{L}$ be a near-linkage of size $k$ in $\Sg$, and $\mathcal{M}$ be a linkage of size $n$ in $\Sg$.
If $\Sg - \mathcal{M}$ is connected, then there is a $\bd$-homeomorphism $\phi: \Sg \to \Sg$ such that $|\phi(\mathcal{L}) \cap \mathcal{M}| \leq k3^{2n+1}$.
\end{theorem}

\begin{proof}
Let $\Sg', \mathcal{L}'$, and $\mathcal{M}'$ be obtained from $\Sg, \mathcal{L}$, and $\mathcal{M}$ by cutting a slightly larger hole $\delta_i'$ around each hole $\delta_i$ of $\Sg$.  We may assume that each $P \in \mathcal{L} \cup \mathcal{M}$ meets each $\delta_i'$ at most twice and that $\mathcal{L} \cap \mathcal{M} \cap \delta_i'=\emptyset$.  By Theorem \ref{bound}, there is a $\bd$-homeomorphism $\phi': \Sg' \to \Sg'$ such that $|\phi'(\mathcal{L}') \cap \mathcal{M}'| \leq k (3^n-1)$. 

For each hole $\delta_i$ we let $\Sg_i$ be the cylinder between $\delta_i'$ and $\delta_i$.  Let $\mathcal{L}_i=\mathcal{L} \cap \Sigma_i$ and $\mathcal{M}_i=\mathcal{M} \cap \Sg_i$.  Note that $\bigcup_i |\mathcal{L}_i| \leq 2k$ and $\bigcup_i |\mathcal{M}_i| \leq 2n$.  By applying Theorem \ref{nearcylinder} to each $\mathcal{L}_i$ and $\mathcal{M}_i$ in $\Sg_i$, and then extending $\phi'$ accordingly, there is a $\bd$-homeomorphism $\phi: \Sg \to \Sg$ such that $|\phi(\mathcal{L}) \cap \mathcal{M}| \leq k (3^n-1) + 2k(3^{2n}-1) \leq k3^{2n+1}$.
\end{proof}

We now further extend Theorem~\ref{bound} to bound the intersection number between a forest and a linkage.
Let $F_1$ and $F_2$ be two forests embedded in $\Sg$.  Robertson and Seymour \cite{gm7} define $F_1$ and $F_2$ to be \e{homotopic} if 

\begin{itemize}
\item
$V(F_1) \cap \bd(\Sg)=V(F_2) \cap \bd(\Sg)$,
\item
for all $s,t \in V(F_1) \cap \bd(\Sg)$, there is a path from $s$ to $t$ in $F_1$ if and only if there is a path from $s$ to $t$ in $F_2$, and
\item
for all $s, t \in V(F_1) \cap \bd(\Sg)$, the $s$-$t$ path in $F_1$ (if it exists) is homotopic to the $s$-$t$ path in $F_2$ (if it exists).  
\end{itemize}

Two forests $F_1$ and $F_2$ are \e{homoplastic} if there is a $\bd$-homeomorphism $\phi$ such that $\phi(F_1)$ is homotopic to $F_2$.

\begin{theorem}
\label{actualexplicitness}
For all $k, n \in \mathbb{N}$ and all surfaces $\Sg$, if $F$ is a forest
in $\Sg$ with $|V(F) \cap \bd(\Sg)| \leq k$, $\mathcal{M}$ is an  $n$-linkage in $\Sg$, and $\Sg -\mathcal{M}$ is connected, then there is a forest $F'$ in $\Sg$ such that $F'$ is homoplastic to $F$ and $|F' \cap \mathcal{M}| \leq 4k(3^{2n+1})$.
\end{theorem}

\begin{proof}
Let $(\Sg, \mathcal{M}, F)$ be a counterexample with $|V(F)|$ minimum.  Since $|V(F)|$ is minimum, all degree 2
vertices of $F$ must be on $\bd(\Sg)$.  Next suppose there is an edge $xy \in E(F)$ such that $x$ has degree $1$ in $F$, and $x \notin \bd(\Sg)$.  Note that contracting $e$ produces a smaller counterexample.  Thus, all leaf vertices of $F$ are on $\bd(\Sg)$.  Let $V_{\geq 3}$ be the vertices of $F$ of degree at least 3, $V_1$ be the leaves of $F$,  and $X$ be the vertices of $F$ not contained on $\bd(\Sg)$.  Since $X \sub V_{\geq 3}$ and all leaves of $F$ are on $\bd(\Sg)$ we have

\[
\sum_{v \in X} d_F(v) \leq \sum_{v \in V_{\geq 3}} d_F(v) < 3|V_1| \leq 3|V(F) \cap \bd(\Sg)|,
\]
where the second to last inequality follows since a forest has average degree less than 2.  

By applying an isotopy we may assume that $\mathcal{M}$ is disjoint from $X$.  For each $x \in X$, let $\D_x$ be a small open disk such that $\D_x$ is disjoint from $\mathcal{M}$.  Let $\Sg':=\Sg - \bigcup_{x \in X} \D_x$.  We transform $F$ into a near-linkage $\mathcal{L}(F)$ on $\Sg'$ as follows.  For each $x \in X$, we split $x$ into $d_F(x)$ copies on $\D_x$ according to the clockwise order of the edges around $x$ in $F$.  Let $\mathcal{M}'$ be the image of $\mathcal{M}$ in $\Sg'$.  Now apply Corollary~\ref{nearlinkage} to $\mathcal{L}(F)$ and $\mathcal{M}'$ in $\Sg'$.  Since $\sum_{v \in X} d_F(v) \leq 3|V(F) \cap \bd(\Sg)|$, it follows that $|V(\mathcal{L}(F))| \leq 4|V(F) \cap \bd(\Sg)|$.  Therefore, there is a $\bd$-homeomorphism $\phi': \Sg' \to \Sg'$ such that $|\phi'(\mathcal{L}(F)) \cap \mathcal{M}'| \leq 4k(3^{2n+1})$.  By  gluing back each $\D_x$ and then contracting each $\D_x$ to a point, we obtain a forest $F'$ in $\Sg$ such that $|F' \cap \mathcal{M}| \leq 4k(3^{2n+1})$ and $F'$ is homoplastic to $F$.
\end{proof}

Theorem \ref{actualexplicitness} is essentially a computable version of \cite[(3.6)]{gm7}, with an explicit value for the constant $w(\Sg,k,n)$.  Unfortunately, we have the additional hypothesis that $\Sg -\mathcal{M}$ is connected.  In the last paragraph
of~\cite{gm13}, it is stated, without proof, that $w(\Sigma,k,n)$ from \cite{gm7} is computable.  Note that the bound $4k(3^{2n+1})$ in Theorem \ref{actualexplicitness} is independent of $\Sigma$.  We conjecture we should also be able to take $w(\Sigma,k,n)=4k(3^{2n+1})$, which would follow from Conjecture \ref{separatingconjecture}.
However, it is important to point out that our proof of Theorem~\ref{protect} does \e{not} rely on the fact that $w(\Sg, k,n)$ is computable.  We will derive Theorem~\ref{protect} from 
Theorem~\ref{bound}.

\section{Linkages on a Cylinder}
The purpose of this section is to establish two lemmas regarding linkages on a cylinder.  Both these lemmas will be used in the proof of Theorem~\ref{protect}.  

It is convenient for us to describe our first lemma in terms of independence in a certain matroid, which we now define.  In general, if $V_1$ and $V_2$ are sets of vertices in a graph $G$, then, for each $A \sub V_1$, the maximum number of disjoint $A$-$V_2$ paths in $G$ is the rank function of a matroid on $V_1$.  We denote the rank function of this matroid as 
$\kappa_{V_1, V_2}$.  

  We will later apply Edmonds' Matroid Intersection Theorem~\cite{edmonds} to two copies of this matroid.  No other knowledge of matroid theory is required, but the interested reader may refer to Oxley~\cite{oxley}.  

Our first lemma is a technical assertion about when we can route  paths across a cylinder given the presence of many other paths.  

\begin{lemma} \label{buffer}
Let $G$ be a graph embedded on a cylinder $\Sg$ with holes $\delta_1$ and $\delta_2$. Let $V_1:=V(G) \cap \delta_1$, $V_2:=V(G) \cap \delta_2$, and  $M$ be the matroid on $V_1$ 
with rank function $\kappa_{V_1, V_2}$.  Let $A_1, B_1, A_2, B_2, \dots, A_n, B_n$ be a cyclically contiguous partition of $V_1$.
If for all $i \in [n]$, $A_i$ is $M$-independent and $r_M(B_i) \geq 2 \sum_{j=1}^n |A_j|$, then $ \bigcup_{j=1}^n A_j$ is $M$-independent. 
\end{lemma}

\begin{proof}
 By hypothesis, for each $i \in[n]$, there exists a collection $\mathcal{A}_i$ of $|A_i|$ disjoint $A_i$-$V_{2}$ paths.  If the paths in $\mathcal{A}:= \bigcup_{i=1}^n \mathcal{A}_i$ are disjoint, we are done.  Otherwise, let $B:=\bigcup_{i=1}^n B_i$.  Since $r_M(B_i) \geq 2 \sum_{j=1}^n |A_j|$ for each $i$, by iteratively augmenting $M$-independent sets, there exists a collection $\mathcal{B}$ of disjoint $B$-$V_{2}$ paths such that 
\begin{itemize}
\item
$|\mathcal{B}|=2 \sum_{i=1}^n |A_i|$, and

\item
For each $i$, $\mathcal{B}$ contains exactly $|A_i|+|A_{i+1}|$ paths with an endpoint in $B_i$ (indices are read modulo $n$). 
\end{itemize}
The idea is to use the paths in $\mathcal{B}$ to reroute the paths in $\mathcal{A}$.  Let $\mathcal{B}_i$ be the paths in $\mathcal{B}$ with an endpoint in $B_i$ and let $m_i:=|A_i|$.  

Label the paths of $\mathcal{A}_1$ as $P_1, \dots, P_{m_1}$ clockwise.  Label the paths of $\mathcal{B}_{1}$ as $R_1, \dots,  R_{m_1+m_2}$ clockwise. Label the paths of $\mathcal{B}_{n}$ as $Q_1, \dots, Q_{m_n+m_1}$ counter-clockwise.   For walks $P$ and $Q$ that intersect, the \e{product} of $P$ with $Q$ is the walk $PQ:=PxQ$, where $x$ is the first vertex of $P$ also in $Q$.  By convention, if $P$ and $Q$ are disjoint $A$-$V_{2}$ paths, the \e{region between $P$ and $Q$} is the (closed) clockwise region in $\Sg$ from $P$ to $Q$. 

We will reroute the paths in $\mathcal{A}_1$ so that they are between $Q_{m_1}$ and $R_{m_1}$. 
Suppose that some path of $\mathcal{A}_1$ is not between $Q_{m_1}$ and $R_{m_1}$.  The crux of the proof is the following claim.

\begin{claim}
Either 
\begin{itemize}
\item
$P_1 \cap Q_{m_1} \neq \emptyset$ and $P_1 Q_{m_1} \cap R_{m_1}=\emptyset$, or
\item
$P_{m_1} \cap R_{m_1} \neq \emptyset$ and $P_{m_1} R_{m_1} \cap Q_{m_1}=\emptyset$.
\end{itemize}
\end{claim}
 
\begin{proof}[Subproof]
If $P_1$ and $P_{m_1}$ are both between $Q_{m_1}$ and $R_{m_1}$, then by planarity, all paths of $\mathcal{A}_1$ would also be, which is a contradiction.  So certainly, $P_1$ or $P_{m_1}$ must intersect $Q_{m_1}$ or $R_{m_1}$.  By symmetry let us assume $P_1$ intersects $Q_{m_1}$ or $R_{m_1}$.  Suppose $P_1$ intersects $Q_{m_1}$.  Then we are done unless 
\[
P_1 Q_{m_1} \cap R_{m_1} \neq \emptyset.
\]
However, this implies that $P_1$ also intersects $R_{m_1}$, and that in fact $P_1$ intersects $R_{m_1}$ \e{before} $Q_{m_1}$.  It follows that $P_{m_1} \cap R_{m_1} \neq \emptyset$ and $P_{m_1} R_{m_1} \cap Q_{m_1}=\emptyset$, as required.   

The remaining case is that $P_1$ intersects $R_{m_1}$, but not $Q_{m_1}$.  Again we have $P_{m_1} \cap R_{m_1} \neq \emptyset$ and $P_{m_1} R_{m_1} \cap Q_{m_1}=\emptyset$. 
\end{proof}

So all paths of $\mathcal{A}_1$ are indeed between $Q_{m_1}$ and $R_{m_1}$ unless
\begin{itemize}
\item
$P_1 \cap Q_{m_1} \neq \emptyset$ and $P_1 Q_{m_1} \cap R_{m_1}=\emptyset$, or
\item
$P_{m_1} \cap R_{m_1} \neq \emptyset$ and $P_{m_1} R_{m_1} \cap Q_{m_1}=\emptyset$.  
\end{itemize}
By symmetry, we may assume $P_1 \cap Q_{m_1} \neq \emptyset$ and $P_1 Q_{m_1} \cap R_{m_1}=\emptyset$.  We replace $P_1$ by $P_1 Q_{m_1}$.  Now, if $P_2, \dots, P_{m_1}$ are all between $Q_{m_1 -1}$ and $R_{m_1}$ then we are done.  Otherwise, by the above claim
\begin{itemize}
\item
 $P_2 \cap Q_{m_1 -1} \neq \emptyset$ and $P_2 Q_{m_1 -1} \cap R_{m_1}=\emptyset$, or
\item
$P_{m_1} \cap R_{m_1} \neq \emptyset$ and $P_{m_1} R_{m_1} \cap Q_{m_1 -1}=\emptyset$. 
\end{itemize}
In the former, we replace $P_2$ by $P_2 Q_{m_1 -1}$.  
In the latter, we replace $P_{m_1}$ by $P_{m_1} R_{m_1}$.  
Note that in both cases the rerouted path is disjoint from $P_1 Q_{m_1}$.  Therefore, we can continue re-routing inductively, until all paths in $\mathcal{A}_1$ are between $Q_{m_1}$ and $R_{m_1}$.  

By repeating the above argument, for each $i \in [n]$ we obtain a family $\mathcal{A}_i'$ of disjoint $A_i$-$V_{2}$ paths such that, for all $i$,
\begin{itemize}
\item
$|\mathcal{A}_i'|=|A_i|$, and
\item
The paths in $\mathcal{A}_i'$ intersect at most $|A_i|$ paths of $\mathcal{B}_i$ and at most $|A_i|$ paths of $\mathcal{B}_{i-1}$.   
\end{itemize}
It immediately follows that the family $\mathcal{A}':=\bigcup_{i=1}^n \mathcal{A}_i'$ is disjoint, since $|\mathcal{B}_i| \geq |A_i|+|A_{i+1}|$ for each $i$.
\end{proof}

We end this section by proving a lemma for linkages in cylindrical grids.  
%The \e{Cartesian product} of two graphs  $G$ and $H$ is the graph $G  \, \square \,  H $ with vertex set $V(G) \times V(H)$, and  where two vertices $(u,u')$ and $(v,v')$ are adjacent in $G  \, \square \,  H $ if and only if  
%\begin{itemize}
%\item
%$u = v$ and $u'$ is adjacent to $v'$ in $H$, or
%\item
%$u' = v'$ and $u$ is adjacent to $v$ in $G$.
%\end{itemize}
Let $C_m$ be a cycle of length $m$ and $P_n$ be a path with $n$ vertices.  The $(m,n)$-\e{cylindrical grid} is the Cartesian product $C_m  \, \square \, P_{n}$.  The two cycles of length $m$ in $C_m  \, \square \,  P_{n}$ that pass through only degree 3 vertices are called the \e{boundary cycles}.  

Suppose that the vertices of a pattern $\Pi$ are contained in a cyclically ordered set (such as a cycle in a graph). We say that $\Pi$ is \e{cross-free}, if there do not exist distinct $a,b,c,d \in V(\Pi)$ such that $\{a,b\}, \{c,d\} \in \Pi$ and the cyclic ordering of $\{a,b,c,d\}$ is $a,c,b,d$ or $a,d,b,c$.  Note that a pattern on a disk is topologically feasible if and only if it is cross-free.  

Our second lemma gives sufficient conditions for finding linkages in cylindrical grids. 

\begin{lemma} \label{cgrid}
Let $G$ be a $(m,n)$-cylindrical grid and let $\Pi$ be a pattern of size $k$ with $V(\Pi)$ contained in a boundary cycle of $G$.  If $\Pi$ is cross-free and $n \geq k$, then $\Pi$ is realizable in $G$.
\end{lemma}

\begin{proof}
If $\Pi$ contains a singleton $\{s\}$, then we can delete $s$ from $G$ and contract the remaining vertices of $\Pi$ one step into the cylinder.  The resulting graph has a $C_{m-1} \, \square \,  P_{n-1}$ minor with $V(\Pi) - \{s\}$ still contained in one of the boundary cycles.  By induction on $k$, we are done.  

Otherwise, since $\Pi$ is cross-free,  we can find an element $\{s, t \} \in \Pi$ and an $s$-$t$ path $P$ of a boundary cycle such that no internal vertex of $P$ is in $V(\Pi)$.  We delete the ends of $P$ and contract the other vertices of $\Pi$ one step into the cylinder.  The resulting graph has a $C_{m-2}  \, \square \,  P_{n-1}$ minor with $V(\Pi) - \{s,t\}$ still contained in one of the boundary cycles.  By induction, we can realize $\Pi - \{\{s,t\}\}$ in $G - \{s,t\}$, and hence we can realize $\Pi$ in $G$.  
\end{proof}

\section{Redundant Vertices on Surfaces} \label{main}
In this section we prove Theorem~\ref{protect}.  Let $G$ be a graph embedded in a surface $\Sg$ and let $\Pi$ be a pattern in $G$.  Recall that a vertex $v \in V(G)$ is \e{$t$-protected in $\Sg$} (with respect to $\Pi$) if

\begin{itemize}

\item
there are $t$ vertex disjoint cycles $C_1, \dots, C_t$ of $G$, bounding discs $\D_1, \dots, \D_t$ in $\Sigma$ with $v \in \D_1 \subset \D_2 \subset \dots \subset \D_t$, and

\item
$V(\Pi)$ is disjoint from $\mathsf{int}(\D_{t})$.

\end{itemize}
We refer to $C_1, \dots, C_t$ as the \e{cycles protecting $v$}.

To apply induction, it turns out to be useful to work with a special kind of surface. To this end, we introduce a `disk with strips'.

 A \e{strip} $S$ is a homeomorph of $[0,1] \times [0,10]$.  The:
 \begin{itemize}
 \item
  \e{ends} of $S$ are the images of $[0,1] \times \{0\}$  and $[0,1] \times \{10\}$;
  
  \item
  \e{equator} of $S$ is the image of $[0,1] \times \{5\}$;
  
  \item
  \e{corners} of $S$ are images of $(0,0), (0,10), (1,0)$, and $(1,10)$.
 \end{itemize}

A \e{disk with $n$ strips} is a surface $\Omega:=\D \cup S_1 \cup \dots \cup S_n$, where $\D$ is a disk and for all distinct $i,j \in [n]$,
\begin{itemize}
\item
$S_i$ is a strip.

\item
$S_i \cap \D$ is the union of the ends of $S_i$.

\item
$S_i$ and $S_j$ are disjoint, except possibly at corners.
\end{itemize}

For example, up to homeomorphism, the only disks with 1 strip are the cylinder and the M\"obius band.
If $\Omega=\D \cup S_1 \cup \dots \cup S_n$ is a disk with $n$ strips, then we say $S_1, \dots, S_n$ are the \e{strips} of $\Omega$ and that $\D(\Omega):=\D$ is the \e{disk} of $\Omega$.  

Let $\Omega$ be a disk with strips, $G$ be a graph embedded in $\Omega$, and $\Pi$ be a pattern in $G$.  We say that a vertex $v \in V(G)$ is \e{$t$-insulated in $\Omega$} (with respect to $\Pi$) if:
\begin{itemize}
\item
there are $t$ vertex disjoint cycles $C_1, \dots, C_t$ of $G$, bounding discs $\D_1, \dots, \D_t$ in $\D(\Omega)$ with $v \in \D_1 \subset \D_2 \subset \dots \subset \D_t=\D(\Omega)$;

\item
$V(\Pi)$ is disjoint from $\mathsf{int}(\D_{t})$; and

\item
each $C_i$ is an induced subgraph of $G \cap \D(\Omega)$.
\end{itemize}

In particular, if we regard $\Omega$ as a surface, then a $t$-insulated vertex is a $t$-protected vertex, but not necessarily vice versa.  

We prove Theorem~\ref{protect} as a corollary of the following theorem.

\begin{theorem} \label{insulate}
For all $k, n \in \N$, there exists a computable constant $\theta:=\theta(k,n) \in \N$ such that if $G$ is a graph embedded in a disk with $n$ strips $\Omega$, $\Pi$ is a pattern in $G$ of size $k$, $v \in V(G)$ is a $\theta(k,n)$-insulated vertex in $\D(\Omega)$ with respect to $\Pi$, and $V(\Pi) \sub \bd(\Omega) \cap \D(\Omega)$, then $v$ is redundant.
\end{theorem}

We also require the following easy lemma which follows from Euler's Formula. See \cite[Proposition 3.6]{lctriangulations} for a proof.

\begin{lemma} \label{bouquet}
Let $\mathcal{C}$ be a family of non-contractible simple closed curves in a surface $\Sg$.  If, for all $C_1,C_2\in \mathcal C$, $C_1\cap C_2=\{b\}$
 and the curves in $\mathcal{C}$ are pairwise non-homotopic (relative to the fixed basepoint $b$), then $|\mathcal{C}| \leq 3g(\Sg)$.  
\end{lemma}

The proof of Theorem~\ref{insulate} is rather lengthy, so we defer it until the next section.  It is however, relatively straightforward to derive Theorem~\ref{protect} from Theorem~\ref{insulate}, which we now proceed to do.  

\begin{reptheorem}{protect}
For all surfaces without boundary $\Sg$ and all $k \in \N$, there exists a computable constant $t:=t(\Sg,k) \in \N$  such that if $G$ is a graph embedded in $\Sg$, $\Pi$ is a $k$-pattern in $G$, and $v \in V(G)$ is a $t$-protected vertex in $\Sg$ with respect to $\Pi$, then $v$ is redundant.
\end{reptheorem}

\begin{proof} For all surfaces without boundary $\Sg$ and all $k \in \N$, define $t(\Sg, k)$ to be $\theta(k, 4k+3g(\Sg))$, where $\theta$ is the function from Theorem~\ref{insulate}.  We will define $\theta$ explicitly in the proof of Theorem~\ref{insulate}, so $t$ is also explicit.  

Let $(G, \Sigma, \Pi, v)$ be a counterexample with $|V(G)|+|E(G)|$ minimal.  That is, $G$ is a graph embedded in a surface $\Sg$, $\Pi$ is a pattern of size $k$ in $G$, and $v \in V(G)$ is a $t$-protected ($t:=t(\Sg,k)$) vertex in $\Sg$ with respect to $\Pi$, yet $v$ is essential.  

Let $C_1, \dots,  C_t$ be cycles protecting $v$, bounding disks $\D_1 \subset \dots \subset \D_t$ in $\Sg$ such that $\sum_{i \in [t]} |V(C_i)|$ is minimum.  Let $\mathcal{L}$ be a $\Pi$-linkage in $G$, and let $H=C_1 \cup \dots \cup C_t$.  

\begin{claim} 
$V(G)=V(H)$.
\end{claim}

\begin{proof}[Subproof]
Suppose not.  First note that $V(\mathcal{L}) \cup V(H)=V(G)$, otherwise we could delete a vertex of $G$ not in $V(\mathcal{L}) \cup V(H)$ to obtain a smaller counterexample.  Next observe that if $e=xy \in E(\mathcal{L})$ and $y \notin V(H)$, then we can contract $e$ onto $x$ to obtain a smaller counterexample.
\end{proof}

Observe that the claim implies that $V(\Pi) \sub V(C_t)$.  

\begin{claim}
Each $C_i$ is an induced subgraph of $G \cap \D_t$. 
\end{claim}

\begin{proof}[Subproof]
Towards a contradiction, suppose that $e \sub \D_t$, $e \notin E(H)$, and $e$ has both of its ends on $C_j$ for some $j \in [t]$.  Note that by minimality, $G$ is simple.  So, there is a cycle $C_j' \sub C_j \cup e$ with length strictly less than $C_j$.  Replacing $C_j$ by $C_j'$ contradicts that $\sum_{i \in [t]} |V(C_i)|$ is minimum.
\end{proof}

We now consider edges $e$ of $G$ not contained in $\D_t$.  We say that such an edge $e$ is \e{contractible} if $e$ and a subpath of $C_t$ bounds a disk in $\Sg$.  Otherwise, $e$ is \e{non-contractible}.    We say that two paths in $\Sg$ are \e{homotopic (relative to $\bd(\D_t)$)} if there is a homotopy between them that always has its endpoints on $\bd(\D_t)$.  

\begin{claim}
There are at most $2k$ homotopy classes of contractible edges.
\end{claim}

\begin{proof}[Subproof]
For each contractible edge $e$, let $P_e$ be a subpath of $C_t$ such that $P_e \cup e$ bounds a disk in $\Sg$.  Observe that $e$ and $f$ are homotopic if and only if $P_e \sub P_f$ or $P_f \sub P_e$.   Now let $\mathcal{E}$ be a collection of contractible edges that are pairwise non-homotopic.   It follows that $\mathcal{P}:=\{P_e: e \in \mathcal{E}\}$ is a collection of pairwise internally disjoint paths of $C_t$.  Also, each $P_e$ must contain an internal vertex which is in $V(\Pi)$, for otherwise we could replace $C_t$ in $H$ by a shorter cycle.  So
\[
|\mathcal{E}| = |\mathcal{P}| \leq |V(\Pi)|=2k. \qedhere
\] 
\end{proof}

\begin{claim}
There are at most $3g(\Sg)$ homotopy classes of non-contractible edges.
\end{claim}

\begin{proof}[Subproof]
Let $\mathcal{N}$ be a collection of non-contractible edges that are pairwise non-homotopic.  Contract the disk $\D_t$ to a point in $b$ in $\Sg$, and let $\mathcal{N}^*$ be the resulting family of curves.  Note that $\mathcal{N}^*$ is now a collection of simple non-contractible closed curves on $\Sg$, each containing $b$ but otherwise pairwise disjoint.  Furthermore, the curves in $\mathcal{N}^*$ are pairwise non-homotopic (relative to the base point $b$).  By Lemma~\ref{bouquet}, there are at most $3g(\Sg)$  such curves.
\end{proof}

By regarding each homotopy class as passing through a distinct strip, we can view $G$ as being embedded on a disk with at most $2k+3g(\Sg)$ strips, $\Omega$, where $\D(\Omega)=\D_t$.  Unfortunately, to apply Theorem~\ref{insulate}, we require $V(\Pi)$ to be on $\bd(\Omega) \cap \D(\Omega)$.  However, if $x \in V(\Pi)$ is not on a corner of a strip of $\Omega$, then we may split a strip in half, and place $x$ at a corner of one of the new strips.  Note that we only need to apply this operation at most $2k$ times.  So, we have shown the following.  

\begin{claim}
$G$ is a graph embedded in a disk with at most $4k+3g(\Sg)$ strips $\Omega'$, $\Pi$ is a pattern in $G$ of size $k$, $v \in V(G)$ is a $\theta(k, 4k+3g(\Sg))$-insulated vertex in $\D(\Omega')$ with respect to $\Pi$, and $V(\Pi) \sub \bd(\Omega') \cap \D(\Omega')$.
\end{claim}

By definition of the function $\theta$, we have that $v$ is indeed redundant for $\Pi$.
\end{proof}

\section{Redundant Vertices on Disks with Strips} \label{technical}
In this section we prove Theorem~\ref{insulate}, which we restate for convenience.  Our proof is based on an unpublished proof of Carl Johnson and Paul Seymour presented at the Workshop on Graph Theory in Oberwolfach, 1999.

\begin{reptheorem}{insulate}
For all $k, n \in \N$, there exists $\theta:=\theta(k,n) \in \N$ such that if $G$ is a graph embedded in a disk with $n$ strips $\Omega$, $\Pi$ is a pattern in $G$ of size $k$, $v \in V(G)$ is a $\theta$-insulated vertex in $\D(\Omega)$ with respect to $\Pi$, and $V(\Pi) \sub \bd(\Omega) \cap \D(\Omega)$, then $v$ is redundant.
\end{reptheorem}

\begin{proof}
Let $\theta(k,n)=\tower(100, 200, \dots, 100n, k100^n)$ and let $m(k,n)=(4n+1)k3^n+8k$.  Note that $\theta(k,0)=k$ for all $k$.  Also, for all $n>0$, an easy induction gives
\[
\theta(k,n) \geq \theta(k + 4m(k,n) (2n+1)^{4n m(k,n)},  n-1)+ 2k+nk3^n.
\]
These are the only two properties of $\theta(k,n)$ that we will use.  Note that $\theta(k,n)$ does not depend on $G$.  

Let $(G, \Omega, \Pi, v)$ be a counterexample with $|E(G)|$ minimal.  Let $n$ be the number of strips in $\Omega$, $k$ the size of $\Pi$, and $\theta=\theta(n,k)$.  Then $v$ is $\theta$-insulated in $\Omega$ with respect to $\Pi$, $V(\Pi)\subseteq \bd(\Omega)\cap \Delta(\Omega)$, and yet $v$ is essential.

Let $\Omega:=\D \cup S_1 \cup \dots \cup S_n$, and let $C_1, \dots, C_\theta$ be cycles insulating $v$, bounding disks $\D_1 \subset \dots \subset \D_{\theta}=\D$.  Let $\mathcal{L}$ be a $\Pi$-linkage in $G$ and let $H=C_1 \cup \dots \cup C_\theta$.  Notice that we may assume $\bd(\Omega) - \Delta_\theta$ is disjoint from $G$.

\begin{claim}
$E(H) \cap E(\mathcal{L})=\emptyset$ and $E(H) \cup E(\mathcal{L})=E(G)$.
\end{claim}

\begin{proof}[Subproof]
Contracting any edges in $E(H) \cap E(\mathcal{L})$ or deleting any edges not in $E(H) \cup E(\mathcal{L})$ would both yield smaller counterexamples.
\end{proof}

\begin{claim}
$V(G)=V(H)$.
\end{claim}

\begin{proof}[Subproof]
Let $xy$ be an edge with $y \notin V(H)$. Since $y \notin V(H)$, $y \notin V(C_\theta)$, and, therefore $y \notin \bd(\Omega)$.  Thus, $G / xy$ is a smaller counterexample. 
\end{proof}

We now examine how $\mathcal{L}$ passes through $\Omega$.  The \e{level} $\ell(x)$ of a vertex $x$ in $C_j$ is defined to be $j$.  Let $P$ be a path with ends $a$ and $b$.  We call $P$ a \e{hill} if

\begin{itemize}

\item
$\ell(a)=\ell(b)$, 

\item
$\ell(c)>\ell(a)$ for all internal vertices $c$ of $P$, and

\item
$P$ and a subpath of $C_{\ell(a)}$ bounds a disk in $\Omega$.

\end{itemize}

Note that if a path $P$ satisfies the first two bullet points and $P \subseteq \D$, then $P$ will automatically satisfy the third bullet point.  However, there may be hills not contained in $\D$.  For example, an edge $xy$ contained in a strip $S$ is a hill if and only if $x$ and $y$ are both on a same end of $S$.

The \e{sea level} $\ell(P)$ of  a hill $P$ is defined to be the level of either of its ends.   Observe there is a subpath $K_P$ of $C_{\ell(P)}$ so that $P\cup K_P$ bounds a disc whose interior is disjoint from the insulated vertex $v$.

\begin{claim} \label{nohill}
$\mathcal{L}$ (as a subgraph) does not contain a hill.
\end{claim}

\begin{proof}[Subproof]
Suppose that $\mathcal{L}$ contains a hill.  Let $\sigma$ be the lowest sea level of all hills of $\mathcal{L}$.   Among all hills of $\mathcal{L}$ at sea level $\sigma$, choose $J$ such that the length of $K_J$ is minimal.   By choice of $J$ we have that $\mathcal{L}$ does not use any internal vertex of $K_J$. Therefore, $(\mathcal{L} \bs E(J)) \cup E(K_J)$ is a $\Pi$-linkage.  Letting $e$ be any edge of $J$, we conclude that $G \bs e$ is a smaller counterexample, a contradiction.
\end{proof}

A path $P=x_0 \dots x_q$ of $G$ is \e{decreasing} if $P \sub \D$ and $\ell(x_0) \leq \dots \leq \ell(x_q)$.  We will require the following claim later.

\begin{claim} \label{decreasing}
Let $A \sub V(C_{\theta})$, and let $i \in [\theta]$.  If there exist $|A|$ disjoint $A$-$C_{i}$ paths in $G \cap \D$, then there exist $|A|$ disjoint decreasing $A$-$C_{i}$ paths in $G \cap \D$.
\end{claim}

\begin{proof}[Subproof]
The proof is similar to the proof of the previous claim.  Let $\mathcal{A}$ be a collection of $|A|$ disjoint $A$-$Z$ paths in $G \cap \D$ with the minimum number of  hills.  We claim that $\mathcal{A}$ is a family of decreasing paths.  Suppose not and let $\sigma$ be the lowest sea level among all hills in $\mathcal{A}$.  Among all hills of $\mathcal{A}$ at sea level $\sigma$, choose $J$ such that the length of $K_J$ is minimal.   By choice of $J$ we have that $\mathcal{A}$ does not use any internal vertex of $K_J$.  Re-routing $\mathcal{A}$ through $K_J$ contradicts the choice of $\mathcal{A}$.
\end{proof}

Let $Y_1, \dots, Y_{\ell}$ be the components of $C_{\theta} - (\bigcup_{i=1}^n \mathsf{int}(\mathsf{ends}(S_i))$.  Define $X_i:=Y_i \cap V(\Pi)$ and observe that 
$X_1, \dots, X_{\ell}$ is a partition $\mathbb{P}$ of $V(\Pi)$ (possibly some $X_i$ are empty).  We say that a path $P$ of $G$ is a \e{nibble} if $P \sub \D$ and the ends of $P$ are in the same part of the partition $\mathbb{P}$.

\begin{claim}
No path of $\mathcal{L}$ is a nibble.
\end{claim}

\begin{proof}[Subproof]
Suppose not, and choose a nibble $L \in \mathcal{L}$ such that $\min \{ i : L \cap C_i \neq \emptyset \}$ is maximum.  By choice of $L$ and planarity, there is a path $K$ of $C_{\theta}$ with the same ends as $L$ such that no path of $\mathcal{L}$ uses an internal vertex of $K$.  By replacing $\mathcal{L}$ by $(\mathcal{L} - \{L\}) \cup \{K\}$ and deleting any edge of $L$ from $G$, we contradict that $G$ is a minimal counterexample.  
\end{proof}

By orienting $C_{\theta}$ clockwise, we may view each part of the partition $\mathbb{P}$ as a linearly ordered set.  For distinct $a, b \in C_{\theta}$, we let $[a,b]$ be the clockwise subpath of $C_{\theta}$ from $a$ to $b$.   Let $\{x_1, \dots, x_p\}$ be one of the parts of the partition (labelled in increasing order).  The key point to keep in mind is that $[x_1, x_p]$ is disjoint from all strips of $\Omega$ (except possibly at corners).  For each $x_i$, let $\mathcal{L}(x_i)$ be the (unique) member of $\mathcal{L}$ starting from $x_i$. Define $\omega(x_i)$ to be the number of protective cycles that $\mathcal{L}(x_i)$ intersects before it uses an edge outside of $\D$.

\begin{claim} \label{omega}
For each $i \in [p]$, $\omega(x_i) \geq \min\{i, p-i+1\}$.  
\end{claim}

\begin{proof}[Subproof]
We proceed by induction on $\min\{i, p-i+1\}$.  Clearly the claim holds for $i \in \{1, p\}$.  Consider an arbitrary $x_i$.  By symmetry we may assume that $i \leq \frac{p}{2}$ and we inductively assume that $\omega(x_{i-1}) \geq i-1$ and $\omega(x_{p-i+2}) \geq i-1$. 

Towards a contradiction assume that $\omega(x_i) \leq i-1$.  Let $a$ be the second vertex of $\mathcal{L}(x_i)$ that is on $C_{\theta}$ ($x_i$ is the first).  Let $Q$ be the subpath of $\mathcal{L}(x_i)$ from $x_i$ to $a$.  Note that $Q \cup [x_i, a]$ and $Q \cup [a,x_i]$ both bound disks in $\D$.  We denote them as $\D_1$ and $\D_2$, respectively.  We say that a region in $\D$ is \e{small} if it does not contain $v$ (the insulated vertex).  Because $\omega(x_i) \leq i-1$, $v$ is not in $\mathcal{L}(x_i)$.  Therefore, exactly one of $\D_1$ or $\D_2$ is small.  There are various cases depending where $a$ lies on $C_{\theta}$ and which of $\D_1$ or $\D_2$ is small.

\begin{scl}
$\D_1$ is not small.
\end{scl}
\begin{proof}[Subproof]
Towards a contradiction assume $\D_1$ is small.  If $a \in [x_i, x_p]$, then $\mathcal{L}$ contains a nibble, a contradiction.  Thus, $a \in [x_p, x_i]$.  Note that  $\omega(x_{p-i+2}) \geq i-1$ by induction.  Since $\omega(x_i) \leq i-1$, the only way to avoid a contradiction is if $\mathcal{L}$ connects $x_i$ to $x_{p-i+2}$ inside $\D$.  However, this path of $\mathcal{L}$ is a nibble, which is also impossible.
\end{proof}

\begin{scl}
$\D_2$ is not small.
\end{scl}
\begin{proof}[Subproof]
Towards a contradiction assume $\D_2$ is small.  If $a \in [x_{i-1}, x_i]$, then $\mathcal{L}$ does not use any internal vertex of $[a, x_i]$.   Therefore,  we can reroute $\mathcal{L}(x_i)$ through $[a,x_i]$, which contradicts that $G$ is a minimal counterexample.   So, $x_{i-1} \in [a, x_i]$.  Since $\omega(x_{i-1}) \geq i-1$, the only way to avoid a contradiction is if $\mathcal{L}(x_i)$ actually connects $x_i$ to $x_{i-1}$ within $\D$.  But then $\mathcal{L}(x_i)$ is a nibble, which is also impossible.
\end{proof}

This completes the proof of the claim, since one of $\D_1$ or $\D_2$ must be small.  Thus, $w(x_i) \geq \min\{i, p-i+1\}$, as required.
\end{proof} 

We now analyze the edges of $G$ not contained in $\D$.   For each strip $S$ let $E(S)$ be the edges of $G$ contained in $S$.

\begin{claim} \label{Smatching}
For each strip $S$, $E(S)$ is a matching with each edge on different ends of $S$.
\end{claim}

\begin{proof}[Subproof]
If $e \in E(S)$ has both ends on a same end of $S$, then $e$ is a hill, which is a contradiction.  If another edge $f \in E(G)$ shares an end with $e$, then $\{e,f\}$ and a subpath $P$ of $C_{\theta}$ bounds a disk in $\Omega$.  If $P$ is just an edge, we may reroute $\mathcal{L}$ through $P$.  If $P$ contains an internal vertex, then $\mathcal{L}$ must contain a hill at sea level $\theta-1$, contradicting Claim~\ref{nohill}.  
\end{proof}

If we regard $\Pi$ as a pattern in $\Omega$ instead of a pattern in $G$, then evidently there is a topological realization of $\Pi$ in $\Omega$, since there is a realization of $\Pi$ in $G$.  Let $\mathcal{M}$ be the topological linkage of size $n$, consisting of the equators of the strips of $\Omega$.  By Theorem~\ref{bound}, there is a topological $\Pi$-linkage $\mathcal{L}'$ such that $|\mathcal{L}' \cap \mathcal{M}| \leq k3^n$.
The pivotal idea is to try and realize $\mathcal{L}'$ in $G$.

Let $m:=(4n+1)k3^n+8k$ and $N:=\theta(k+4m(2n+1)^{4nm}, n-1)$.  Observe that $\theta(k,n)=N+2k+nk3^n$.  We set $M$ to be the matroid on $V(C_\theta)$ with rank function $\kappa_{V(C_\theta),V(C_N)}$.

For each strip $S$ of $\Omega$, we let $V(S)$ be the vertices covered by $E(S)$.  By Claim~\ref{Smatching}, we may partition  $V(S)$ as $V_0(S) \cup V_1(S)$, according to the end of $S$ a vertex belongs to.  For $i=0,1$, we let $M_i(S)$ be the restriction of $M$ to $V_i(S)$ respectively. 
We may use the matching $E(S)$ to identify a vertex in $V_0(S)$ with a vertex in $V_1(S)$; in this way, we may regard $M_0(S)$ and $M_1(S)$ as matroids on the
same ground set.  For $X \sub V_0(S)$ we let $\clone(X)$ be the copy of $X$ in $V_1(S)$.

Recall that $m=(4n+1)k3^n+8k$.  We first consider the case when $M_0(S)$ and $M_1(S)$ have a large common independent set, for each strip $S$ of $\Omega$.

\begin{case} 
For each strip $S$ of $\Omega$, $M_0(S)$ and $M_1(S)$ have a common independent set of size $m$.
\end{case}

\begin{claim} \label{indroots}
Each part of the partition $\mathbb{P}$ of $V(\Pi)$ is independent in $M$.
\end{claim}

\begin{proof}[Subproof]
Label the vertices of an arbitrary part $X$ of $\mathbb{P}$ as $x_1, \dots, x_p$ (clockwise).  Choose an arbitrary strip $S$, and let $I$ be an $M_0(S)$-independent subset of size $p$.  By Claim~\ref{decreasing}, there is a family $\mathcal{Q}$ of $p$ disjoint decreasing $I$-$C_N$ paths.  Label these paths as $Q_1, \dots, Q_p$ (counter-clockwise).  We will use $\mathcal{Q}$ to construct $p$ disjoint $X$-$C_N$ paths in $G \cap \D$.  By Claim~\ref{omega}, for each $i \in [p]$, $w(x_i) \geq \min\{i, p-i+1\}$. 

So for each $i \in \{1, \dots, \lceil p / 2 \rceil \}$ we can define a path $\mathcal{P}(x_i)$ as follows:
\begin{itemize}
\item
Follow $\mathcal{L}(x_i)$ until it intersects $C_{\theta -(i-1)}$.

\item
Follow $C_{\theta-(i-1)}$ (counter-clockwise) until intersecting $Q_{\lceil p / 2 \rceil - (i-1)}$.

\item
Follow $Q_{\lceil p / 2 \rceil - (i-1)}$ until reaching $C_N$.
\end{itemize}
For $i \in \{p, p-1, \dots, \lceil p / 2 \rceil  +1\}$  we define $\mathcal{P}(x_i)$ as follows:
\begin{itemize}
\item
Follow $\mathcal{L}(x_i)$ until it intersects $C_{\theta-p+i}$.

\item
Follow $C_{\theta-p+i}$ (clockwise) until intersecting $Q_{\lceil p / 2 \rceil +p-i+1}$.

\item
Follow $Q_{\lceil p / 2 \rceil +p-i+1}$ until reaching $C_N$.
\end{itemize}
Since all three portions of these paths are decreasing, it follows that
\[
\mathcal{P}:=\{ \mathcal{P}(x_i) : i \in [p] \}
\]
is a family of disjoint $X$-$C_N$ paths.
\end{proof}

Next we show that $V(\Pi)$ is actually $M$-independent.  In fact, we prove the following much stronger claim.

\begin{claim} \label{Aind}
For each strip $S_i$ of $\Omega$ there exists a subset $K_i$ of $V_0(S_i)$ of size $k3^n$ such that $V(\Pi) \cup \bigcup_{i \in [n]} (K_i \cup \clone(K_i))$ is independent in $M$.
\end{claim}

\begin{proof}[Subproof]
Of course we are in the case when $M_0(S_i)$ and $M_1(S_i)$ have a large common independent set for each strip $S_i$ of $\Omega$.
So, for each $i \in [n]$ let $J_i$ be an independent set of size $(4n+1)k3^n+8k$ in $M_0(S_i)$, such that $\clone(J_i)$ is also independent in $M_1(S_i)$.  We partition $J_i$ into three sets $J_i^1, J_i^2$ and $J_i^3$ where $J_i^1$ are the first $2(nk3^n+2k)$  points, $J_i^2$ are the middle $k3^n$ points and $J_i^3$ are the last $2(nk3^n+2k)$ points.  We will apply Lemma~\ref{buffer} to the two collections of sets
\[
\mathcal{A}:=\{J_i^2 : i \in [n] \} \cup \{\clone(J_i^2): i \in [n] \} \cup \{X_i : i \in [l] \},
\]
and
\[
\mathcal{B}:=\{J_i^k : i \in [n], k \in \{1,3\} \} \cup \{\clone(J_i^k) : i \in [n], k \in \{1,3\} \}.
\]
Observe that each set in $\mathcal{A}$ is indeed $M$-independent, and that for any $B \in \mathcal{B}$ we have
\[
r_M(B) =2(nk3^n+2k) = 2\sum_{A \in \mathcal{A}} |A|.
\]
Therefore, by Lemma~\ref{buffer}, we conclude that $\bigcup_{A \in \mathcal{A}}A$ is $M$-independent.  Setting $K_i=J_i^2$ for each $i \in [n]$ gives the result.
\end{proof}

We can now attempt to realize the topological linkage $\mathcal{L}'$ in $G$.  We may assume that $\mathcal{L}'$ intersects $\bd(\D)$ only at vertices in $\mathcal{A}$.  Let $G':=G - \mathsf{int}(\D_N)$.  By removing all the strips from $\Omega$ and keeping track of how the paths in $\mathcal{L}'$ pass through the strips, we are left with a $\Pi'$-linkage problem in the disk $\D$, where $V(\Pi') \sub V(\mathcal{A})$.  

By Claim~\ref{Aind}, we have that $V(\mathcal{A})$ is $M$-independent.  Therefore, by Claim~\ref{decreasing}, there exists a family of $|V(\mathcal{A})|$ disjoint decreasing $V(\mathcal{A})$-$C_N$ paths in $G'$.  These decreasing paths, together with the protective circuits $C_{\theta}, C_{\theta-1}, \dots, C_N$ form a large cylindrical-grid minor $H'$ in $G' \cap \D$.  Since 
\[
\theta -N \geq 2k+nk3^n =|V(\mathcal{A})| \geq |\Pi'|,
\]
Lemma~\ref{cgrid} implies that $G' \cap \D$ actually has a $\Pi'$-linkage.  It follows that $G'$ has a $\Pi$-linkage, and that $v$ is redundant for $\Pi$ in $G$ since $v \notin V(G')$, completing the proof in Case 1.

The remaining case is if $M_0(S)$ and $M_1(S)$ do not have a large common independent set, for some strip $S$ of $\Omega$.  By re-indexing, we may assume that $S=S_1$.

\begin{case}
$M_0(S_1)$ and $M_1(S_1)$ do not have a common independent set of size $m$.
\end{case}

The idea in this case is to reduce the number of strips.  Since $M_0(S_1)$ and $M_1(S_1)$ do not have a common independent set of size $m$, by the Matroid Intersection Theorem~\cite{edmonds}, there is a partition $\{A,B\}$ of $V_0(S_1)$ such that
\[
r_{M_0(S_1)}(A) + r_{M_1(S_1)} (\clone(B)) < m. 
\]
That is, there exist subsets $T$ and $U$ of $V(G \cap \D)$ such that 
\begin{itemize}
\item
$T$ separates $A$ from $V(C_N)$ in $G \cap \D$,

\item
$U$ separates $\clone(B)$ from $V(C_N)$ in $G \cap \D$, and

\item
$|T|+|U| < m$.
\end{itemize}

We choose such a $T$ and $U$ with $|T \cup U|$ minimum.  We then choose an index $\gamma \in \{\theta-1, \dots, \theta - m\}$ such that $T \cup U$ is disjoint from $C_\gamma$.  Recall that the level of a vertex $x \in G \cap \D$ is the unique index $j$ such that $x \in V(C_j)$.

A path is a $\D_{\gamma}$-path if both its ends belong on $\D_{\gamma}$, and it is otherwise disjoint from $\D_{\gamma}$.  Evidently, a $\D_{\gamma}$-path must have both of its ends on $C_{\gamma}$.  For each path $P$ of $\mathcal{L}$, we define $\mathcal{U}(P)$ to be the family of maximal $\D_{\gamma}$-subpaths of $P$.  We then define $\mathcal{U}(\mathcal{L}):=\bigcup_{P \in \mathcal{L}} \mathcal{U}(P)$.   
\begin{claim} \label{homotopy}
There are at most $(2n+1)^{4nm}$ homotopy classes of paths in  $\mathcal{U}(\mathcal{L})$. 
\end{claim}
\begin{proof}[Subproof]
Let $Q \in \mathcal{U}(\mathcal{L})$.  Since $Q$ does not contain any hills, there is no subpath $K$ of $C_{\gamma}$ such that $Q \cup K$ bounds a disk in $\Omega$.  In particular, this implies that $Q$ must use an edge outside of $\D$ and that the homotopy class of $Q$ is determined by how $Q$ passes through the strips of $\Omega$.    Let $\mathcal{A}$ be the alphabet $\{S_1, \dots, S_n, S_1^{-1}, \dots, S_n^{-1} \}$.  If we orient each strip of $\Omega$, then the homotopy class of $Q$, denoted $\mathcal{H}(Q)$, is then naturally encoded by a string of letters from $\mathcal{A}$.  We make the convention that if $S_iS_i^{-1}$ or $S_i^{-1}S_i$ appears in $\mathcal{H}(Q)$ for some $i \in [n]$, then we cancel it.  With this convention, we prove that each letter of $\mathcal{A}$ appears at most $2m$ times in $\mathcal{H}(Q)$, from which the claim follows.

Towards a contradiction assume that some letter $\alpha$ appears at least $2m+1$ times in $\mathcal{H}(Q)$.  By reversing the direction of $Q$ if necessary, we may assume $\alpha=S$, for some strip $S$.  Let $e_1, \dots, e_{2m+1}$ be edges of $Q$ corresponding to the occurrences of $S$ in $\mathcal{H}(Q)$.  Let $e_i=w_i x_i$ so that $Q$ traverses $e_i$ from $w_i$ to $x_i$ and so that this traversal is consistent with the orientation of $S$.  By cancellation, the next edge of $Q$ after $e_i$ that is outside $\D$ cannot pass through $S$ in the backward direction.  We re-index so that $x_1, \dots, x_{2m+1}$ occur clockwise along one end of the strip $S$ (this is not necessarily their order in $Q$).

Either $x_{m+1}$ occurs before $x_{m+2}$ along $Q$ or vice versa.  By symmetry, we assume the former.  Let $Q':=x_{m+1}Q$ and let $y$ be the first vertex of $Q'$ such that the next edge of $Q'$ after $y$ passes through a strip.  By cancellation, it follows that $y \in [x_{2m+1},x_1]$. 

Recall that a region $\mathcal{R}$ in $\D$ is \e{small} if it does not contain the insulated vertex $v$.  Clearly, either $Q'y \cup [y,x_{m+1}]$ bounds a small region, or $Q'y \cup [x_{m+1},y]$ bounds a small region $\mathcal{R}$.  So, we either have $\{x_1, \dots, x_{m+1}\} \sub \mathcal{R}$ or $\{x_{m+1}, \dots, x_{2m+1}\} \sub \mathcal{R}$.  In either case we get a contradiction, since $Q'y$ intersects at most $\theta - \gamma \leq m$ insulating cycles. 
\end{proof}

We call a homotopy class of $\mathcal{U}(\mathcal{L})$ \e{thin} if it has size at most $4m$, otherwise it is \e{thick}.

\begin{claim} \label{dichotomy}
Either there are at most $n-1$ thick homotopy classes of $\mathcal{U}(\mathcal{L})$ (up to inversion),  or
$T \cup U$ separates $V(C_{\theta})$ from $V(C_N)$.
\end{claim}

\begin{proof}[Subproof]
Let $\mathcal{H}$ be a thick homotopy class, represented as a string of letters from $\{S_1, \dots, S_n, S_1^{-1}, \dots, S_n^{-1} \}$.  Note that $\mathcal{H}$ is not the empty string since $\mathcal{L}$ has no hills.  Suppose $\mathcal{H}$ is of length at least 2.  Consider an arbitrary path $Q \in \mathcal{H}$ and let $e_1$ and $e_2$ be the edges of $Q$ that correspond to the first two letters of the homotopy class of $Q$.  For $i \in [2]$, let $e_i=x_i y_i$, so that $Q$ traverses $e_i$ from $x_i$ to $y_i$.  Finally, let $Q'$ be the subpath of $Q$ from $y_1$ to $x_2$.  If $\mathcal{H}$ is not thin, then the collection $\mathcal{H}':=\{Q' : Q \in \mathcal{H}\}$ has size at least $4m+1$.  Therefore, there exists $J \in \mathcal{H}'$ and some subpath $K$ of $C_{\theta}$ such that $J \cup K$ bounds a small region that contains at least $2m$ members of $\mathcal{H}$.  This is a contradiction, as each path in $\mathcal{H}'$ intersects at most $\theta - \gamma \leq m$ insulating cycles.   

Thus, if $\mathcal{H}$ is thick, it must be a string of length 1.  Up to inversion, this implies that $\mathcal{H}=S$, for some strip $S$, leaving at most $n$ possibilities for $\mathcal{H}$.  
However, consider the homotopy class $\mathcal{H}_1$ represented by the string $S_1$.  If $\mathcal{H}_1$ is not thick we are done, so assume that $\mathcal{H}_1$ contains more than $4m$ paths.  Therefore, $\mathcal{H}_1$ contains a collection of at least $2m$ vertex-disjoint paths.  Observe that each of these paths must pass through $V_0(S_1)$ and $V_1(S_1)$.  Therefore, there is a subset $X$ of $V_0(S_1)$ of size $2m$ such that 
\[
\kappa_{G \cap \D} (X, V(C_{\gamma}))=2m=\kappa_{G \cap \D} (\clone(X), V(C_{\gamma})).
\]
Note that, for the partition $\{A,B\}$ of $V_0(S_1)$, we have that $|X \cap A| \geq m$ or $|X \cap B| \geq m$.  By symmetry, we assume the former.  Since $|T| <m$, we conclude that $A$ is still connected to $V(C_{\gamma})$ in $(G \cap \D) - T$.  Since $V(C_{\gamma})$ contains no vertices of $T$, and $T$ separates $A$ from $V(C_N)$ in $G \cap \D$, it follows that $T \cap \D_{\gamma}$ must separate $V(C_{\gamma})$ from $V(C_{N})$ in $G \cap \D_{\gamma}$.  By the minimality of $| T \cup U|$ it follows that $U=\emptyset$ and that $T \cap \D_{\gamma}=T$.  This completes the proof of the claim.   
\end{proof}

We handle the first possibility of Claim~\ref{dichotomy} first.  

\begin{scase}
 There are at most $n-1$ thick homotopy classes of $\mathcal{U}(\mathcal{L})$ (up to inversion).
 \end{scase}

Let $G':=(G \cap \D_{\gamma}) \cup \mathcal{U}(\mathcal{L})$.  By Claim~\ref{homotopy} we can regard $G'$ as embedded in a disk with at most $\beta:=(2n+1)^{4m}$  strips 

We describe how to reduce the $\Pi$-linkage problem in $G$ to a $\Pi'$-linkage problem in $G'$.  
Let $P \in \mathcal{L}$.  If $P$ has a vertex in $C_{\gamma}$, then let $x$ be the first such vertex and let $y$ be the last.  If they exist, place $\{x,y\}$ into $\Pi'$ and repeat for all paths in $\mathcal{L}$.  By splitting strips if necessary, we may assume that $G'$ is embedded in a disk with at most $\beta' \leq \beta + 2k$ strips
\[
\Omega':=\D_{\gamma} \cup S_1' \cup \dots S_{\beta'}',
\]
and with $V(\Pi') \sub \bd(\Omega')$.

At first glance it seems as if we have increased the complexity of our problem, since we have more strips than we began with.  However, at most $n-1$ of the strips $S_1', \dots, S_{\beta'}'$ are thick.  By re-indexing, we may assume that $S_n', \dots, S_{\beta'}'$ are all thin.  By deleting all the edges contained in $S_n' \cup \dots \cup S_{\beta'}'$, and keeping track of how the paths in $\mathcal{L}$ pass through $S_n' \cup \dots \cup S_{\beta'}'$, we reduce to a $\Pi''$-linkage in $\Omega'':=\D_{\gamma} \cup S_1' \cup \dots \cup S_{n-1}'$, where $|\Pi''| \leq k + 4m (2n+1)^{4nm}$.  Since $v$ is a $\gamma$-insulated vertex with respect to $\Pi''$, and $\gamma  \geq \theta(k + 4m(2n+1)^{4nm},  n-1)$, it follows that $v$ is redundant for $\Pi''$, and hence also for $\Pi$.  This completes the subcase.

We now handle the remaining subcase.

\begin{scase}
$T \cup U$ separates $V(C_{\theta})$ from $V(C_N)$ in $G \cap \D$.
\end{scase}

We will reduce the $\Pi$-linkage problem in $G$ to a $\Pi'$-linkage problem in $G \cap \D_N$.  We do this by  proving that $|V(\mathcal{L}) \cap V(C_{N})|$ is small.  So, let $x \in V(\mathcal{L}) \cap V(C_{N})$, and suppose $x  \in V(P)$ for $P \in \mathcal{L}$.  We define $\mathsf{next}(x)$ to be the next vertex of $P$ that is also in $T \cup U$ (we allow $\mathsf{next}(x)=x$).  The first thing to observe  is that $\mathsf{next}(x)$ does exist.  This follows since $T \cup U$ separates $V(C_{\theta})$ from $V(C_N)$.  Secondly, since $\mathcal{L}$ contains no hills, the map $x \mapsto \mathsf{next}(x)$ is injective.  So, 
\[
|V(\mathcal{L}) \cap V(C_{N})| \leq |T \cup U| < m.
\]
By keeping track of how the paths in $\mathcal{L}$ enter and leave $\D_N$, we reduce to a $\Pi'$-linkage problem in $G \cap \D_N$, where $|\Pi'|<m$.  Since $N \geq \theta(m,0)$, we have that $v$ is redundant for $\Pi'$ in $G \cap \D_N$, and hence redundant for $\Pi$ in $G$.  

This completes the subcase, and hence the entire proof.
\end{proof}

\subsection*{Acknowledgements}
We would like to thank Jiří Matoušek, Eric Sedgwick, Martin Tancer, and Uli Wagner for making an early version of \cite{untangling} available to us.

\bibliography{references}{}

\begin{thebibliography}{10}

\bibitem{tightplanar}
Isolde Adler, Stavros~G. Kolliopoulos, Philipp~Klaus Krause, Daniel Lokshtanov,
  Saket Saurabh, and Dimitrios Thilikos.
\newblock Tight bounds for linkages in planar graphs.
\newblock In {\em Automata, languages and programming. {P}art {I}}, volume 6755
  of {\em Lecture Notes in Comput. Sci.}, pages 110--121. Springer, Heidelberg,
  2011.

\bibitem{edmonds}
Jack Edmonds.
\newblock Submodular functions, matroids, and certain polyhedra.
\newblock In {\em Combinatorial {S}tructures and their {A}pplications ({P}roc.
  {C}algary {I}nternat. {C}onf., {C}algary, {A}lta., 1969)}, pages 69--87.
  Gordon and Breach, New York, 1970.

\bibitem{simpleralgorithm}
Ken-ichi Kawarabayashi and Paul Wollan.
\newblock A simpler algorithm and shorter proof for the graph minor
  decomposition [extended abstract].
\newblock In {\em S{TOC}'11---{P}roceedings of the 43rd {ACM} {S}ymposium on
  {T}heory of {C}omputing}, pages 451--458. ACM, New York, 2011.

\bibitem{lickorish}
W.~B.~R. Lickorish.
\newblock A finite set of generators for the homeotopy group of a
  {$2$}-manifold.
\newblock {\em Proc. Cambridge Philos. Soc.}, 60:769--778, 1964.

\bibitem{lctriangulations}
Aleksander Malni{\v{c}} and Bojan Mohar.
\newblock Generating locally cyclic triangulations of surfaces.
\newblock {\em J. Combin. Theory Ser. B}, 56(2):147--164, 1992.

\bibitem{untangling}
Ji{\v{r}}{\'{\i}} Matou{\v{s}}ek, Eric Sedgwick, Martin Tancer, and Uli Wagner.
\newblock Untangling two systems of noncrossing curves.
\newblock In {\em Graph drawing}, volume 8242 of {\em Lecture Notes in Comput.
  Sci.}, pages 472--483. Springer, Cham, 2013.

\bibitem{mazoit2013}
Fr{\'e}d{\'e}ric Mazoit.
\newblock A single exponential bound for the redundant vertex theorem on
  surfaces.
\newblock {\em arXiv preprint arXiv:1309.7820}, 2013.

\bibitem{oxley}
James Oxley.
\newblock {\em Matroid theory}, volume~21 of {\em Oxford Graduate Texts in
  Mathematics}.
\newblock Oxford University Press, Oxford, second edition, 2011.

\bibitem{gm7}
Neil Robertson and P.~D. Seymour.
\newblock Graph minors. {VII}. {D}isjoint paths on a surface.
\newblock {\em J. Combin. Theory Ser. B}, 45(2):212--254, 1988.

\bibitem{gm13}
Neil Robertson and P.~D. Seymour.
\newblock Graph minors. {XIII}. {T}he disjoint paths problem.
\newblock {\em J. Combin. Theory Ser. B}, 63(1):65--110, 1995.

\bibitem{gm16}
Neil Robertson and P.~D. Seymour.
\newblock Graph minors. {XVI}. {E}xcluding a non-planar graph.
\newblock {\em J. Combin. Theory Ser. B}, 89(1):43--76, 2003.

\bibitem{gm20}
Neil Robertson and P.~D. Seymour.
\newblock Graph minors. {XX}. {W}agner's conjecture.
\newblock {\em J. Combin. Theory Ser. B}, 92(2):325--357, 2004.

\bibitem{gm22}
Neil Robertson and Paul Seymour.
\newblock Graph minors. {XXII}. {I}rrelevant vertices in linkage problems.
\newblock {\em J. Combin. Theory Ser. B}, 102(2):530--563, 2012.

\end{thebibliography}
\bibliographystyle{plain}

\end{document}